\documentclass{amsart}

\newtheorem{theorem}{Theorem}[section]
\newtheorem{lemma}[theorem]{Lemma}
\newtheorem{proposition}[theorem]{Proposition}

\theoremstyle{definition}

\newtheorem{example}[theorem]{Example}
\newtheorem{conjecture}[theorem]{Conjecture}

\theoremstyle{remark}

\numberwithin{equation}{section}

\begin{document}

\title[Algebraic Montgomery-Yang Problem]{Algebraic Montgomery-Yang Problem: the non-cyclic case}
%\title[Homological Algebraic Montgomery-Yang Problem]{Homological Algebraic Montgomery-Yang Problem}

%    Information for first author
\author{DongSeon Hwang}
%    Address of record for the research reported here
\address{Department of Mathematical Sciences, KAIST, Daejon, Korea}
%    Current address
\curraddr{School of Mathematics, Korea Institute For Advanced
Study, Seoul 130-722, Korea} \email{dshwang@kias.re.kr}
%    \thanks will become a 1st page footnote.
\thanks{Research supported by Basic Science Research Program through the National Research Foundation(NRF)
of Korea funded by the Ministry of education, Science and
Technology (2010-0001652).}

%    Information for second author
%\author{Author Two}
\author{JongHae Keum}
\address{School of Mathematics, Korea Institute For Advanced Study, Seoul 130-722, Korea}
\email{jhkeum@kias.re.kr}
%\thanks{Support information for the second author.}

%    General info
\subjclass[2000]{Primary 14J17}
%\subjclass[2000]{Primary 14J17, 14E20; Secondary 46E25, 20C20}

%\date{April 2009, November 12, 2007 and, in revised form, January 30, 2007.}

%\dedicatory{This paper is dedicated to our advisors.}

\keywords{Algebraic Montgomery-Yang Problem, $\mathbb{Q}$-homology
projective plane, quotient singularity, orbifold
Bogomolov-Miyaoka-Yau inequality}

\begin{abstract}
Montgomery-Yang problem predicts that every pseudofree circle
action on the 5-dimensional sphere has at most $3$ non-free
orbits. Using a certain one-to-one correspondence, Koll\'ar
formulated the algebraic version of the Montgomery-Yang problem:
every projective surface $S$ with quotient singularities such that
the second Betti number $b_2(S) = 1$ has at most $3$ singular
points if its smooth locus $S^0$ is simply connected.

We prove the conjecture under the assumption that $S$ has at least
one non-cyclic singularity. In the course of the proof, we
classify projective surfaces $S$ with quotient singularities such
that (i) $b_2(S) = 1$, (ii) $H_1(S^0, \mathbb{Z}) = 0$, and (iii)
$S$ has 4 or more singular points, not all cyclic, and prove that
all such surfaces have $\pi_1(S^0)\cong \mathfrak{A}_5$, the
icosahedral group.
\end{abstract}

\maketitle

%\section*{This is an unnumbered first-level section head}
%This is an example of an unnumbered first-level heading.

%% The correct journal style for \specialsection is all uppercase; a known bug
%% in amsart.cls prevents this, so input must be uppercase until it is fixed.
%\specialsection*{This is a Special Section Head}
%\specialsection*{THIS IS A SPECIAL SECTION HEAD}
%This is an example of a special section head%
%%%%%%%%%%%%%%%%%%%%%%%%%%%%%%%%%%%%%%%%%%%%%%%%%%%%%%%%%%%%%%%%%%%%%%%%
%\footnote{Here is an example of a footnote. Notice that this footnote
%text is running on so that it can stand as an example of how a footnote
%with separate paragraphs should be written.
%\par
%And here is the beginning of the second paragraph.}%
%%%%%%%%%%%%%%%%%%%%%%%%%%%%%%%%%%%%%%%%%%%%%%%%%%%%%%%%%%%%%%%%%%%%%%%%

\section{Introduction}

A pseudofree $\mathbb{S}^1$-action on a sphere $\mathbb{S}^{2k-1}$
is a smooth $\mathbb{S}^1$-action which is free except for
finitely many non-free orbits (whose isotropy types
$\mathbb{Z}_{m_1}, \ldots, \mathbb{Z}_{m_n}$ have pairwise
relatively prime orders).

For $k=2$ Seifert \cite{Sei} showed that such an action must be
linear and hence has at most two non-free orbits. In the contrast
to this, for $k=4$ Montgomery and Yang \cite{MY} showed that given
any pairwise relatively prime collection of positive integers
$m_1,\ldots, m_n$, there is a pseudofree $\mathbb{S}^1$-action on
homotopy 7-sphere whose non-free orbits have exactly those orders.
Petrie \cite{Pet} proved similar results in all higher odd
dimensions. This led Fintushel and Stern to formulate the
following problem:

\begin{conjecture} [\cite{FS87}](Montgomery-Yang Problem) \\
{\it Let $$\mathbb{S}^1 \times \mathbb{S}^5 \rightarrow
\mathbb{S}^5$$ be a pseudofree $\mathbb{S}^1$-action.
 Then it has at most $3$ non-free orbits.}
\end{conjecture}

The problem has remained unsolved since its formulation.

Pseudofree $\mathbb{S}^1$-actions on 5-manifolds $L$ have been
studied in terms of the 4-dimensional quotient orbifold
$L/\mathbb{S}^1$ (see e.g., \cite{FS85}, \cite{FS87}). A manifold
is called a \emph{ rational homology sphere} if it has the same
$\mathbb{Q}$-homology groups with a sphere, i.e., it has the same
Betti numbers with a sphere. The following one-to-one
correspondence was known to Montgomery, Yang, Fintushel and Stern,
and recently observed by Koll\'ar (\cite{Kol05}, \cite{Kol08}):

\begin{theorem}[cf. \cite{Kol05}, \cite{Kol08}]\label{1-1}
There is a one-to-one correspondence between:
\begin{enumerate}
\item Pseudofree $\mathbb{S}^1$-actions on $5$ dimensional
rational homology spheres $L$ with $H_1(L, \mathbb{Z}) = 0$.
\item Smooth, compact $4$-manifolds $M$ with boundary such that
\begin{enumerate}
\item $\partial{M} = \cup_i L_i$ is a disjoint union of lens spaces $L_i = \mathbb{S}^3/\mathbb{Z}_{m_i}$,
\item the $m_i$ are relatively prime to each other,
\item $H_1(M, \mathbb{Z}) = 0$ and $H_2(M, \mathbb{Z}) \cong \mathbb{Z}$.
\end{enumerate}
\end{enumerate}
Furthermore, $L$ is diffeomorphic to $\mathbb{S}^5$ iff $\pi_1(M)
= 1$.
\end{theorem}

We recall that a normal projective surface with the same Betti
numbers with $\mathbb{C}\mathbb{P}^2$ is called a \emph{rational
homology projective plane} or a \emph{$\mathbb{Q}$-homology
projective plane} or a \emph{$\mathbb{Q}$-homology
$\mathbb{C}\mathbb{P}^2$}. When a normal projective surface $S$
has quotient singularities only, $S$ is a $\mathbb{Q}$-homology
$\mathbb{C}\mathbb{P}^2$ if the second Betti number $b_2(S)=1$.

It is known that a $\mathbb{Q}$-homology projective plane with
quotient singularities has at most 5 singular points (cf.
\cite{HK} Corollary 3.4). Recently, the authors have classified
$\mathbb{Q}$-homology projective planes with 5 quotient
singularities (\cite{HK}, also see \cite{K10}).

Using the one-to-one correspondence of Theorem \ref{1-1}, Koll\'ar
formulated the algebraic version of the Montgomery-Yang problem as
follows:

\begin{conjecture}[\cite{Kol08}]\label{amy} (Algebraic Montgomery-Yang Problem)\\
{\it Let $S$ be a $\mathbb{Q}$-homology projective plane with
quotient singularities. Assume that $S^0:=S \backslash Sing(S)$ is
simply connected. Then $S$ has at most $3$ singular points.}
\end{conjecture}

In this paper, we verify the conjecture when $S$ has at least one
non-cyclic singularity. More precisely, we prove the following:

\begin{theorem}\label{main}
Let $S$ be a $\mathbb{Q}$-homology projective plane with quotient
singularities such that $\pi_1(S^0)= \{1\}$. Assume that $S$ has
at least one non-cyclic singularity. Then $|Sing(S)|\le 3$.
\end{theorem}

We note that the condition $\pi_1(S^0)= \{1\}$ cannot be replaced
by the weaker condition $H_1(S^0, \mathbb{Z}) = 0$. There are
infinitely many examples of $\mathbb{Q}$-homology projective
planes with exactly $4$ quotient singularities, where three of
them are cyclic and one of them is non-cyclic, such that $H_1(S^0,
\mathbb{Z}) = 0$ but $\pi_1(S^0) \neq \{1\}$ (\cite{Brieskorn} or
\cite{Kol08}, Example 31). These examples are the global quotients
$$S_{I_m}:=\mathbb{CP}^2/I_{m}=(\mathbb{CP}^2/Z)/\mathfrak{A}_5,$$
where $I_m\subset GL(2,\mathbb{C})$ is the group of order $120m$
in Brieskorn's list (see Table 1), an extension of the icosahedral
group $\mathfrak{A}_5\subset PGL(2,\mathbb{C})$ by the cyclic
group $Z\cong\mathbb{Z}_{2m}$, and the action of $I_m$ on
$\mathbb{CP}^2$ is induced from the natural action on
$\mathbb{C}^2$. We call $S_{I_m}$ a \emph{Brieskorn quotient}.

 On the other hand, it follows from the
orbifold Bogomolov-Miyaoka-Yau inequality that every
$\mathbb{Q}$-homology projective plane $S$ with quotient
singularities such that $H_1(S^0, \mathbb{Z}) = 0$ has at most $4$
singular points(cf. \cite{Kol08}, \cite{HK}). Therefore, to prove
Theorem \ref{main}, it is enough to classify $\mathbb{Q}$-homology
projective planes $S$ with $4$ quotient singularities, not all
cyclic, such that $H_1(S^0, \mathbb{Z}) = 0$. It turns out that
such a surface is isomorphic to a Brieskorn quotient.

\begin{theorem}\label{class}
Let $S$ be a $\mathbb{Q}$-homology projective plane with $4$
quotient singularities, not all cyclic, such that $H_1(S^0,
\mathbb{Z}) = 0$. Then the following hold true.
\begin{enumerate}
\item $S$ has $3$ cyclic singularities of type $\mathbb{C}^2/\mathbb{Z}_2$, $\mathbb{C}^2/\mathbb{Z}_3$,
$\mathbb{C}^2/\mathbb{Z}_5$, and one non-cyclic singularity of
type $\mathbb{C}^2/I_{m}$, where $I_m\subset GL(2,\mathbb{C})$ is
the $2m$-ary icosahedral group of order $120m$ $($in Brieskorn's
notation$)$. Furthermore, the $3$ cyclic singularities are of type
$\frac{1}{2}(1,1)$, $\frac{1}{3}(1,\alpha)$,
$\frac{1}{5}(1,\beta)$, if the $3$ branches of the dual graph of
the non-cyclic singularity are of type $\frac{1}{2}(1,1)$,
$\frac{1}{3}(1,3-\alpha)$, $\frac{1}{5}(1,5-\beta)$ $($see Table
4$)$.
\item $-K_S$ is ample.
\item The minimal resolution of $S$ can be obtained by starting with a minimal
rational ruled surface and blowing up inside $3$ of the fibres,
i.e. the blowing up starts at three centers, the intersection
points of the $3$ fibres with a section.
\item $S$ is isomorphic to the Brieskorn quotient $\mathbb{CP}^2/I_{m}$,
where $I_m$ is determined by the non-cyclic singularity of $S$ and
its action on $\mathbb{CP}^2$ is induced by the natural action on
$\mathbb{C}^2$.
\item $\pi_1(S^0)\cong \mathfrak{A}_5$, the alternating group of order $60$.
\end{enumerate}
\end{theorem}

In the proof, we use the orbifold Bogomolov-Miyaoka-Yau inequality
(Theorem \ref{bmy} and \ref{bmy2}) and a detailed computation for
$(-1)$-curves on the minimal resolution $S'$ of $S$. The latter
idea was used in \cite{K08}.

In the cyclic case (where $S$ has cyclic singularities only),
Conjecture \ref{amy} has been confirmed in a separate paper
\cite{HK2} unless $S$ is a rational surface with $K_S$ ample.

\medskip
Throughout this paper, we work over the field $\mathbb{C}$ of
complex numbers.

\section{Algebraic surfaces with quotient singularities}
\subsection{} A singularity $p$ of a normal surface $S$ is called a quotient
singularity if the germ  is locally analytically isomorphic to
$(\mathbb{C}^2/G,O)$  for some nontrivial finite subgroup $G$ of
$GL_2(\mathbb{C})$ without quasi-reflections. Brieskorn classified
such finite subgroups of $GL(2, \mathbb{C})$ [Bri]. Table 1
summarizes the result. Here we only explain the notation for dual
graph.

$$
\begin{array}{lcl}
 <q,q_1> &:=& \text{ the dual graph of the singularity of type }
 \dfrac{1}{q}(1,q_1)
\\
<b;s_1, t_1; s_2, t_2;s_3, t_3> &:=& \text{ the tree of the form}\\
& &\\
& &
\begin{picture}(60,30)
\put(40,25){$<s_2, t_2>$}
 \put(60,15){\line(0,0){6}}
\put(0,5){$<s_1,t_1> -\underset{-b}\circ - <s_3, t_3>$}
\end{picture}
\end{array}
$$
\\
For more information about the table, we refer to the original
paper of Brieskorn\cite{Brieskorn}.

\footnotesize

\begin{table}[ht]
\label{eqtable}
\renewcommand\arraystretch{1.5}
\caption{Classification of finite subgroups of $GL(2, \mathbb{C})$}
\noindent $$
\begin{array}{|c|c|c|c|l|}
\hline
$Type$&G&|G|& G/[G, G]&$Dual Graph $ \Gamma_G\\
\hline
A_{q, q_1}&C_{q, q_1}&q&\mathbb{Z}_q&<q,q_1>\\ & & & & 0<q_1<q$, $ (q,q_1)=1\\
\hline D_{q, q_1}&(Z_{2m}, Z_{2m};D_q, D_q) & 4mq
&\mathbb{Z}_{2m}\times \mathbb{Z}_2&<b;2,1;2,1;q,q_1>\\ & & &
&m=(b-1)q-q_1$
odd$\\
 \hline
D_{q, q_1}&(Z_{4m}, Z_{2m};D_q,
C_{2q})&4mq&\mathbb{Z}_{4m}&<b;2,1;2,1;q,q_1> \\
& & & & m=(b-1)q-q_1$ even$\\
\hline
T_m& (Z_{2m}, Z_{2m};T,T)&24m&\mathbb{Z}_{3m}&<b;2,1;3,2;3,2>,  m=6(b-2)+1\\
& &&&<b;2,1;3,1;3,1>,  m=6(b-2)+5\\
\hline
T_m& (Z_{2m}, Z_{2m};T,D_2)&24m&\mathbb{Z}_{3m}&<b;2,1;3,1;3,2>,  m=6(b-2)+3\\
\hline

& &&&<b;2,1;3,2;4,3>,  m=12(b-2)+1\\
O_m& (Z_{2m}, Z_{2m};O,O)&48m&\mathbb{Z}_{2m}&<b;2,1;3,1;4,3>,  m=12(b-2)+5\\
& &&&<b;2,1;3,2;4,1>,  m=12(b-2)+7\\
& &&&<b;2,1;3,1;4,1>,  m=12(b-2)+11\\
\hline

& &&&<b;2,1;3,2;5,4>,  m=30(b-2)+1\\
& &&&<b;2,1;3,2;5,3>,  m=30(b-2)+7\\
& &&&<b;2,1;3,1;5,4>,  m=30(b-2)+11\\
I_m& (Z_{2m}, Z_{2m};I,I)&120m&\mathbb{Z}_{m}&<b;2,1;3,2;5,2>,  m=30(b-2)+13\\
& &&&<b;2,1;3,1;5,3>,  m=30(b-2)+17\\
& &&&<b;2,1;3,2;5,1>,  m=30(b-2)+19\\
& &&&<b;2,1;3,1;5,2>,  m=30(b-2)+23\\
& &&&<b;2,1;3,1;5,1>,  m=30(b-2)+29\\
\hline
\end{array}
$$
\end{table}

\normalsize

\subsection{}  Let $S$ be a normal projective surface with quotient singularities and $$f : S'
\rightarrow S$$ be a minimal resolution of $S$. It is well-known
that quotient singularities are log-terminal
 singularities. Thus one can write $$K_{S'} \underset{num}{\equiv} f^{*}K_S -
 \sum_{p \in Sing(S)}{\mathcal{D}_p},$$ where $\mathcal{D}_p = \sum(a_jE_j)$ is
 an effective $\mathbb{Q}$-divisor supported on $f^{-1}(p)=\cup E_j$ with $0 \leq a_j < 1$ for each singular point $p$.
It implies that
\[K^2_S = K^2_{S'} - \sum_{p}{\mathcal{D}_p^2}= K^2_{S'} + \sum_{p}{\mathcal{D}_pK_{S'}}.
\]

\begin{lemma}[]\label{-k} If $-K_S$ is ample, then $C^2\ge -1$ for any irreducible curve $C\subset S'$ not contracted by $f : S'
\rightarrow S$.
\end{lemma}

\begin{proof} Note that $C(f^{*}K_S)< 0$ and $C(\sum{\mathcal{D}_p})\ge 0$. Thus $CK_{S'}<0$, and hence $C^2\ge -1$.
\end{proof}

Also we  recall the orbifold Euler characteristic
$$ e_{orb}(S) := e(S) - \sum_{p \in Sing(S)} \Big ( 1-\frac{1}{|G_p|} \Big ),$$
where $G_p$ is the local fundamental group of $p$.

The following theorem, called the orbifold Bogomolov-Miyaoka-Yau
inequality, is one of the main ingredients in the proof of our
main theorem.

\begin{theorem}[\cite{Sakai}, \cite{Miyaoka}, \cite{KNS}, \cite{Megyesi}]\label{bmy} Let
$S$ be a normal projective surface with quotient singularities
such that $K_S$ is nef. Then
\[
K_{S}^2 \leq 3e_{orb}(S).
\]
In particular, $$ 0\leq e_{orb}(S).$$
\end{theorem}

The weaker inequality holds when $-K_S$ is nef.

\begin{theorem}[\cite{KM}]\label{bmy2} Let
$S$ be a normal projective surface with quotient singularities such
that $-K_S$ is nef. Then $$ 0\leq e_{orb}(S).$$
\end{theorem}

\subsection{} Let $S$ be a normal projective surface with quotient singularities and $f : S'
\rightarrow S$ be a minimal resolution of $S$. It is well-known
that the torsion-free part of the second cohomology group,
$$H^2(S', \mathbb{Z})_{free} := H^2(S', \mathbb{Z})/(torsion)$$  has a
lattice structure which is unimodular. For a quotient singular
point $p\in S$, let $$R_p\subset H^2(S', \mathbb{Z})_{free}$$ be
the sublattice of $H^2(S', \mathbb{Z})_{free}$ spanned by the
numerical classes of the components of $f^{-1}(p)$. It is a
negative definite lattice, and its discriminant group
$${\rm disc}(R_p):={\rm Hom}(R_p, \mathbb{Z})/R_p$$ is isomorphic to the
abelianization $G_p/[G_p, G_p]$ of the local fundamental group
$G_p$. In particular, the absolute value $|\det(R_p)|$ of the
determinant
 of the intersection matrix of $R_p$ is equal to the order
$|G_p/[G_p, G_p]|$. Let
 $$R= \oplus_{p \in Sing(S)} R_p\subset H^2(S', \mathbb{Z})_{free}$$ be the sublattice of $H^2(S', \mathbb{Z})_{free}$ spanned by the
numerical classes of the exceptional curves of $f:S'\to S$. We
also consider the sublattice $$R+\langle K_{S'}\rangle\subset
H^2(S', \mathbb{Z})_{free}$$ spanned by $R$ and the canonical
class $K_{S'}$. Note that $${\rm rank}(R)\le {\rm rank}(R+\langle
K_{S'}\rangle)\le{\rm rank}(R)+1.$$

\begin{lemma}[\cite{HK}, Lemma 3.3]\label{detR} Let $S$ be a normal projective surface with quotient singularities and $f : S'
\rightarrow S$ be a minimal resolution of $S$. Then the following
hold true.
\begin{enumerate}
\item ${\rm rank}(R+\langle K_{S'}\rangle)={\rm
rank}(R)$ if and only if $K_S$ is numerically trivial.
\item $\det(R+\langle K_{S'}\rangle)=\det(R)\cdot K_S^2$ if $K_S$ is not numerically trivial.
\item If in addition $b_2(S)=1$ and $K_S$ is not numerically trivial, then
$R+\langle K_{S'}\rangle$ is a sublattice of finite index in the
unimodular lattice $H^2(S', \mathbb{Z})_{free}$, in particular
$|\det(R+\langle K_{S'}\rangle)|$ is a nonzero square number.
\end{enumerate}
\end{lemma}

We denote the number $|\det(R+\langle K_{S'}\rangle)|$ by $D$,
i.e., we define
$$D:=|\det(R+\langle K_{S'}\rangle)|.$$

The following will be also used in our proof.

\begin{lemma}[]\label{coprime} Let $S$ be a $\mathbb{Q}$-homology projective plane with quotient singularities
 such that $H_1(S^0,\mathbb{Z}) = 0$. Let $f:S' \rightarrow S$ be a minimal resolution. Then
\begin{enumerate}
\item $H^2(S', \mathbb{Z})$ is torsion free, i.e. $H^2(S', \mathbb{Z})=H^2(S',
\mathbb{Z})_{free}$,
\item $R$ is a primitive
sublattice of the unimodular lattice $H^2(S', \mathbb{Z})$, \item
 ${\rm disc}(R)$ is a cyclic group, in particular, the orders $|G_p/[G_p, G_p]| = |\det(R_p)|$ are pairwise relatively prime, \item $K_S$ is not numerically
 trivial, i.e. $K_S$ is either ample or anti-ample,
\item $D=|\det(R)|
K_S^2$ and $D$ is a nonzero square number,
 \item the Picard group $Pic(S')$ is generated over $\mathbb{Z}$ by the
exceptional curves and a $\mathbb{Q}$-divisor $M$ of the form
$$M = \frac{1}{\sqrt{D}} f^*K_S +  \underset{p \in Sing(S)}{\sum}
b_p e_p $$ for some integers $b_p$, where $e_p$ is a generator of
${\rm disc}(R_p)$.
\end{enumerate}
\end{lemma}

\begin{proof} (1), (2) and (3) are easy to see (cf. \cite{Keum}, Proposition 2.3 and Lemma 3.4).

(4) Assume that $K_S$ is numerically trivial.
 Then $S'$
is an Enriques surface if all singularities are rational double
points, and is a rational surface otherwise. If $S'$ is an
Enriques surface, then $H_1(S^0, \mathbb{Z}) \neq 0$ since
$H_1(S', \mathbb{Z}) = \mathbb{Z}/2$ (cf. Proposition 2.3 in
\cite{Keum}). Thus $S$ is a rational surface, and
$$K_{S'} \underset{num}{\equiv} -\underset{p \in Sing(S)}{\sum} \mathcal{D}_p$$
with $\mathcal{D}_p\underset{num}{\not\equiv} 0$ for some $p$.
Note that $\mathcal{D}_p$ defines an element of $R_p^*:={\rm
Hom}(R_p, \mathbb{Z})$ and the discriminant group ${\rm
disc}(R_p):=R_p^*/R_p$ has order $|\det(R_p)|$. Thus\\
$|\det(R_p)|\mathcal{D}_p\in R_p$ but $\mathcal{D}_p\notin R_p$ if
$\mathcal{D}_p\underset{num}{\not\equiv} 0$. Now we  see that
$$\big(\underset{p}{\prod}
|\det(R_p)| \big) K_{S'} \in R \subset H^2(S', \mathbb{Z}),$$ but
$K_{S'} \notin R$. Hence  the primitive closure $\bar{R}$ of $R$
in $H^2(S', \mathbb{Z})$ is not equal to $R$. Now by Lemma 2.5 in
\cite{Keum}, $H_1(S^0, \mathbb{Z}) \neq 0$.

(5) follows from (4) and Lemma \ref{detR}.

(6) Note first that $Pic(S')=H^2(S', \mathbb{Z})$ and the
sublattice $R\subset H^2(S', \mathbb{Z})$ generated by the
exceptional curves is a primitive sublattice of corank 1. Let
$R^\perp\subset H^2(S', \mathbb{Z})$ be the orthogonal complement
of $R$. Note that $R^\perp$ is positive definite and of rank $1$.
Since $H^2(S', \mathbb{Z})$ is unimodular,
$$\det(R^\perp) =|\det(R)|= \underset{p \in Sing(S)}{\prod} |\det(R_p)|.$$
Note that $f^*K_S \in R^\perp$. Thus $R^\perp$ is generated by
$$v:=\frac{|\det(R)|}{\sqrt{D}}f^* K_S,$$ and ${\rm disc}(R^\perp)$ is
generated by $$\frac{1}{\sqrt{D}}f^* K_S.$$ Also note that
$${\rm disc}(R^\perp\oplus R)\cong (\mathbb{Z}/|\det(R)|)\oplus(\mathbb{Z}/|\det(R)|).$$
Thus $Pic(S')/(R^\perp\oplus R)$ is an isotropic subgroup of ${\rm
disc}(R^\perp\oplus R)$ of order $|\det(R)|$, hence is generated
by an element $M\in {\rm disc}(R^\perp\oplus R)$ of order
$|\det(R)|$. Moreover $M$ is the sum of a generator of ${\rm
disc}(R^\perp)$ and a generator of ${\rm disc}(R)$, since
$Pic(S')$ is unimodular. By replacing $M$ by $kM$ for a suitable
choice of an integer $k$, we get $M$ of the desired form
 $$M = \frac{1}{\sqrt{D}} f^*K_S +  \underset{p \in Sing(S)}{\sum}
b_p e_p$$ for some integers $b_p$ with $0\le b_p < |\det(R_p)|$,
where ${\sum} b_p e_p$ is a generator of ${\rm disc}(R)$. This
proves that $Pic(S')$ is generated over $\mathbb{Z}$ by $R$, $v$
and $M$. Finally, note that $$|\det(R)|M=v\quad ({\rm mod}\,\,
R),$$ i.e., $v$ is generated by $M$ and $R$. Thus $Pic(S')$ is
generated over $\mathbb{Z}$ by $R$ and $M$.
\end{proof}

\section{Proof of Theorem \ref{class}}

Let $S$ be a $\mathbb{Q}$-homology projective plane with 4 or more
quotient singularities with $H_1(S^0, \mathbb{Z}) = 0$. By Lemma
\ref{coprime}(3), the orders of the abelianized local fundamental
groups are pairwise relatively prime. Thus by Theorem \ref{bmy2},
one can see that $S$ has $4$ singular points and the $4$-tuple of
orders of the local fundamental groups must be one of the
following:

\begin{enumerate}
\item $(2,3,5,q)$, $q \geq 7$,
\item $(2,3,7,q)$, $11 \leq q \leq 41$,
\item $(2,3,11,13)$.
\end{enumerate}

Table 1 shows that all non-cyclic singularities of type different
from $I_m$ have abelianized local fundamental groups of order
divisible by 2 or 3.

Assume that one of the singularities is non-cyclic. By Lemma
\ref{coprime}(3), it must be of type $I_m$ and the other 3
singularities are cyclic of order 2, 3 and 5, respectively. Here
we recall that $I_m\subset GL(2,\mathbb{C})$ is the $2m$-ary
icosahedral group of order $120m$. Table 1 shows that there are
$8$ infinite cases of type $I_m$.

There are two types of order 3, $<3,2>$ and $<3,1>$; three types
of order 5, $<5,4>$, $<5,3>\cong <5,2>$ and $<5,1>$.
 Thus there are exactly $48$ infinite cases for
 possible combinations of types of singularities. That is, there
are exactly $48$ infinite cases for $R$, the sublattice of
$Pic(S')=H^2(S', \mathbb{Z})$ generated by all exceptional curves,
where $f:S'\to S$ is a minimal resolution. In each of the 48 cases
we compute $D=|\det(R)|K_S^2$ and check if $D$ is a square number
(see Lemma \ref{coprime}(5)), using elementary number theoretic
arguments. There remain 8 infinite cases and 2 sporadic cases, as
given in Table \ref{det1} and Table \ref{det2}. In both tables,
the entries of the column $b$ are the possible values of $b$ that
make $D$ a square number.

We will explain how to compute $D$. First note that $$|\det(R)|=2
\cdot 3\cdot 5\cdot m=30m.$$  To compute $K^2_S$, we use the
equality from (2.2)
\[K^2_S = K^2_{S'} + \sum_{p}{\mathcal{D}_pK_{S'}}.
\]
Note that $S'$ has $H^1(S', \mathcal{O}_{S'})=H^2(S',
\mathcal{O}_{S'})=0$. Thus by Noether formula,
\[K^2_{S'}= 12- e(S')=10-b_2(S')=9-\mu,\]
where $\mu$ is the number of the exceptional curves of $f$.\\ For
each singular point $p$, the coefficients of the
$\mathbb{Q}$-divisor $\mathcal{D}_p$ can be obtained by
 solving the equations given by the adjunction formula
$$\mathcal{D}_pE=-K_{S'}E=2+E^2$$ for each exceptional curve
$E\subset f^{-1}(p)$. Once we know the coefficients, we can easily
compute the intersection number $\mathcal{D}_pK_{S'}$.

\small

\begin{table}[ht]
\caption{}\label{det1}
\renewcommand\arraystretch{1.5}
 \noindent\[
\begin{array}{|c|c|c|}
\hline
\textrm{Type of $R$} & D=|det(R)|K^2_S  & b\\
 \hline

<2,1>+<3,2>+<5,4>+<b;2,1;3,2;5,4> & 180(5b^2-50b+79) & $none$ \\
<2,1>+<3,2>+<5,4>+ <b;2,1;3,2;5,3> & 180(5b^2-36b+48) &  b =8\\
<2,1>+<3,2>+<5,4>+  <b;2,1;3,1;5,4> &180(5b^2-40b+52) & $none$ \\
<2,1>+<3,2>+<5,4>+<b;2,1;3,2;5,2> & 180(5b^2-34b+41) & $none$ \\
<2,1>+<3,2>+<5,4>+ <b;2,1;3,1;5,3> & 180(5b^2-26b+27) &  $none$ \\
<2,1>+<3,2>+<5,4>+ <b;2,1;3,2;5,1> & 180(5b^2-20b+18) &  $none$ \\
<2,1>+<3,2>+<5,4>+<b;2,1;3,1;5,2> &180(5b^2-24b+22) &   $none$ \\
<2,1>+<3,2>+<5,4>+ <b;2,1;3,1;5,1> &900(b-1)^2  & b \geq 2\\
\hline

<2,1>+<3,2>+<5,3>+<b;2,1;3,2;5,4> & 36(25b^2-190b+277)  &  $none$ \\
<2,1>+<3,2>+<5,3>+ <b;2,1;3,2;5,3> & 36(25b^2-120b+134) &  $none$ \\
<2,1>+<3,2>+<5,3>+ <b;2,1;3,1;5,4> &  36(25b^2-140b+162)  &  $none$ \\
<2,1>+<3,2>+<5,3>+<b;2,1;3,2;5,2> & 36(25b^2-110b+111) &  $none$ \\
<2,1>+<3,2>+<5,3>+ <b;2,1;3,1;5,3> &  36(5b-7)^2  & b \geq 2\\
<2,1>+<3,2>+<5,3>+ <b;2,1;3,2;5,1> &  36(25b^2-40b+8)  &  $none$ \\
<2,1>+<3,2>+<5,3>+<b;2,1;3,1;5,2> &   36(5b-6)^2  & b \geq 2\\
<2,1>+<3,2>+<5,3>+ <b;2,1;3,1;5,1> & 36(25b^2+10b-37)  & $none$  \\
\hline

<2,1>+<3,2>+<5,1>+<b;2,1;3,2;5,4> & 36(25b^2-130b+159)  &  $none$ \\
<2,1>+<3,2>+<5,1>+ <b;2,1;3,2;5,3> & 36(25b^2-60b+28)  & $none$  \\
<2,1>+<3,2>+<5,1>+ <b;2,1;3,1;5,4> &  36(5b-8)^2  & b \geq 2 \\
<2,1>+<3,2>+<5,1>+<b;2,1;3,2;5,2> & 36(25b^2-50b+17)  &  $none$ \\
<2,1>+<3,2>+<5,1>+ <b;2,1;3,1;5,3> & 36(25b^2-10b-37)  &  $none$ \\
<2,1>+<3,2>+<5,1>+ <b;2,1;3,2;5,1> & 36(25b^2+20b-74)  & $none$  \\
<2,1>+<3,2>+<5,1>+<b;2,1;3,1;5,2> &  36(25b^2-38)  &  $none$ \\
<2,1>+<3,2>+<5,1>+ <b;2,1;3,1;5,1> & 36(25b^2+70b-99)  & $none$  \\
\hline
\end{array}
\]
\end{table}

\begin{table}[ht]
\caption{} \label{det2}
\renewcommand\arraystretch{1.5}
\noindent\[
\begin{array}{|c|c|c|}
\hline
\textrm{Type of $R$} & D=|det(R)|K^2_S  & b\\
 \hline

<2,1>+<3,1>+<5,4>+<b;2,1;3,2;5,4> &20(45b^2-390b+593)  & $none$  \\
<2,1>+<3,1>+<5,4>+ <b;2,1;3,2;5,3> &20(45b^2-264b+326)  & $none$ \\
<2,1>+<3,1>+<5,4>+  <b;2,1;3,1;5,4> & 100(9b^2-60b+74)  &  $none$ \\

<2,1>+<3,1>+<5,4>+<b;2,1;3,2;5,2> &20(45b^2-246b+275)  &  $none$ \\
<2,1>+<3,1>+<5,4>+ <b;2,1;3,1;5,3> &20(45b^2-174b+157)  &  $none$ \\
<2,1>+<3,1>+<5,4>+ <b;2,1;3,2;5,1> &100(3b-4)^2  & b \geq 2 \\
<2,1>+<3,1>+<5,4>+<b;2,1;3,1;5,2> & 20(45b^2-156b+124)  & $none$  \\
<2,1>+<3,1>+<5,4>+ <b;2,1;3,1;5,1> & 20(45b^2-30b-17)  & $none$  \\
\hline

<2,1>+<3,1>+<5,3>+<b;2,1;3,2;5,4> &4(225b^2-1410b+1903)  & $none$ \\
<2,1>+<3,1>+<5,3>+ <b;2,1;3,2;5,3> &4(15b-26)^2  &  b \geq 2\\
<2,1>+<3,1>+<5,3>+ <b;2,1;3,1;5,4> & 4(225b^2-960b+968)  & $none$ \\
<2,1>+<3,1>+<5,3>+<b;2,1;3,2;5,2> & 4(15b-23)^2  &  b \geq 2\\
<2,1>+<3,1>+<5,3>+ <b;2,1;3,1;5,3> & 4(225b^2-330b+11)  & $none$ \\
<2,1>+<3,1>+<5,3>+ <b;2,1;3,2;5,1> & 4(225b^2-60b-338)  &  $none$ \\
<2,1>+<3,1>+<5,3>+<b;2,1;3,1;5,2> &  4(225b^2-240b-46)  &  $none$ \\
<2,1>+<3,1>+<5,3>+ <b;2,1;3,1;5,1> &  4(225b^2+390b-643)  &  $none$ \\
\hline

<2,1>+<3,1>+<5,1>+<b;2,1;3,2;5,4> &4(15b-29)^2  &  b \geq 2\\
<2,1>+<3,1>+<5,1>+ <b;2,1;3,2;5,3> & 4(225b^2-240b-278)  & $none$  \\
<2,1>+<3,1>+<5,1>+ <b;2,1;3,1;5,4> & 4(225b^2-420b+86)  & $none$  \\
<2,1>+<3,1>+<5,1>+<b;2,1;3,2;5,2> &4(225b^2-150b-317)  &  $none$ \\
<2,1>+<3,1>+<5,1>+ <b;2,1;3,1;5,3> &4(225b^2+210b-763)  &  $none$ \\
<2,1>+<3,1>+<5,1>+ <b;2,1;3,2;5,1> & 4(225b^2+480b-1076)  & b=2\\
<2,1>+<3,1>+<5,1>+<b;2,1;3,1;5,2> & 4(225b^2+300b-712)  &   $none$ \\
<2,1>+<3,1>+<5,1>+ <b;2,1;3,1;5,1> & 4(225b^2+930b-1201)  &  $none$ \\
\hline
\end{array}
\]
\end{table}

\normalsize

We first rule out the two sporadic cases.

\begin{lemma}[]\label{spor1} The case $<2,1>+<3,2>+<5,4>+ <8;2,1;3,2;5,3>$ does not occur.
\end{lemma}

\begin{proof} In this case, $m=30(b-2)+7=187$, so
$$|\det(R)|=30\cdot 187.$$
The number of exceptional curves $\mu=13$, so $K_{S'}^2=-4$, where
$f:S' \rightarrow S$ is a minimal resolution. Let
$p_1,p_2,p_3,p_4$ be the four singular points. Let $E_1, \ldots,
E_6$ be the components of $f^{-1}(p_4)$ such that

\medskip
 $$
\begin{picture}(100,40)
\put(0,25){$\overset{-2}E_2-\overset{-2}E_3-\overset{-8}E_6-\overset{-2}E_5-\overset{-3}E_4$}
 \put(60,15){\line(0,0){6}}
 \put(55,5){$\underset{-2}{E_1}$}
\end{picture}
$$

\medskip\noindent
is their dual graph. Solving the equations given by the adjunction
formula, we get
$$K_{S'} = f^*K_S - \frac{93E_1 + 186E_6+ 62E_2 + 124E_3 + 112E_4 +
149E_5}{187}.$$  It is easy to compute that
$$K_S^2= K^2_{S'}+ \frac{186E_6K_{S'}+ 112E_4K_{S'}}{187}=-4+\frac{186\cdot 6+ 112}{187}=\frac{480}{187}.$$
Thus $$D=|\det(R)|K_S^2=120^2.$$ Note that $K_S^2> 3e_{orb}(S)$,
so $-K_S$ is ample by the orbifold Bogomolov-Miyaoka-Yau
inequality. Thus $S'$ is a rational surface, not minimal. Also
note that the divisor $M$ from Lemma \ref{coprime}(6) takes the
form
$$M=\frac{1}{120}f^*K_S
+\underset{p \in Sing(S)}{\sum} a_p e_p.$$ Let $C$ be a
$(-1)$-curve on $S'$. By Lemma \ref{coprime}(6), $C$ can be
written as
$$C= kM +r$$ for some integer
$k$ and some $r\in R$, hence as $$C = \frac{k}{120}f^*K_S+ C(1)
+C(2) +C(3) +C(4),$$ where $C(i)$ is a $\mathbb{Q}$-divisor
supported on $f^{-1}(p_i)$. Note that
$$C^2=(\frac{k}{120}f^*K_S)^2+ C(1)^2 +C(2)^2 +C(3)^2 +C(4)^2.$$
Since $(f^*K_S)C(i)=0$ for all $i$, we have
$$(f^*K_S)C= (f^*K_S)(\frac{k}{120}f^*K_S)=\frac{k}{120}K_S^2=\frac{4k}{187}.$$ Since
$-K_S$ is ample and $C\notin R$, we see that $(f^*K_S)C<0$, hence
$k<0.$ Note that $K_{S'}C=-1$. From the equality
$$K_{S'}C=(f^*K_S)C -\frac{(93E_1 + 186E_6+ 62E_2 + 124E_3
+ 112E_4 + 149E_5)C}{187},$$ we get $$(93E_1 + 186E_6+ 62E_2 +
124E_3 + 112E_4 + 149E_5)C=187+4k.$$ This is possible only if
$$E_6C=E_5C=E_4C=E_3C=0,\quad E_2C=E_1C=1,\quad k=-8.$$ Since $E_jC(4)=E_jC$ for $j=1,...,6$,
we obtain the coefficients of $C(4)$ by solving the equations
given by the above intersection numbers. $$C(4)=-\frac{106E_1 +
133E_2 + 79E_3 + 5E_4 + 15E_5+ 25E_6}{187}=E_1^*+E_2^*,$$ where
$E_j^*\in {\rm Hom}(R_{p_4}, \mathbb{Z})$ is the dual vector of
$E_j$. Thus $$C(4)^2=(E_1^*+E_2^*)C(4)=-\frac{106+133}{187}.$$ Now
we have
$$\underset{j\le 3}{\sum} C(j)^2 =C^2-C(4)^2-(\frac{-8f^*K_S}{120})^2=-1+\frac{239}{187}
-\frac{32}{15\cdot 187}>0$$ which contradicts  the negative
definiteness of exceptional curves.
\end{proof}

\begin{lemma}[]\label{spor2} The case $<2,1>+<3,1>+<5,1>+ <2;2,1;3,2;5,1>$ does not occur.
\end{lemma}

\begin{proof}  The proof is similar to the previous case. In this case, $m=19$ and $\mu=8$, so $|\det(R)|=30\cdot 19$ and $K_{S'}^2=1$.
Let $B_2, B_3$ be the components of $f^{-1}(p_2), f^{-1}(p_3)$.
Let $E_1, \ldots, E_5$ be the components of $f^{-1}(p_4)$ such
that

\medskip
 $$
\begin{picture}(100,40)
\put(0,25){$\overset{-2}E_2-\overset{-2}E_3-\overset{-2}E_5-\overset{-5}E_4$}
 \put(60,15){\line(0,0){6}}
 \put(55,5){$\underset{-2}{E_1}$}
\end{picture}
$$

\medskip\noindent
is their dual graph. Then
$$K_{S'} = f^*K_S - \frac{B_2}{3} - \frac{3B_3}{5} - \frac{9E_1 + 6E_2 + 12E_3 + 15E_4 +
18E_5}{19},$$
$$K_S^2=\frac{28\cdot 56}{15\cdot 19}, \quad
D=|\det(R)|K_S^2=56^2.$$ Here again by the orbifold
Bogomolov-Miyaoka-Yau inequality,  $-K_S$ is ample and $S'$ is a
rational surface, not minimal. Let $C$ be a $(-1)$-curve on $S'$.
Then
$$C= \frac{k}{56}f^*K_S+ C(1) +C(2) +C(3) +C(4)$$ for some integer $k$ and some
$\mathbb{Q}$-divisor $C(i)$ supported on $f^{-1}(p_i)$.\\
Since $(f^*K_S)C=\frac{28k}{285}<0$, we see that $k<0$ and we get
$$95B_2C + 171B_3C + 15(9E_1 + 6E_2 + 12E_3 + 15E_4 +
18E_5)C=285+28k.$$ This is impossible because $k <0$ and $E_jC
\geq 0, B_iC \geq 0$ for every $i, j$.
\end{proof}

\begin{lemma}[]\label{-ample} For any of the $8$ infinite cases, $-K_S$ is ample.
\end{lemma}

\begin{proof} For the $8$ infinite cases, we compute $K_S^2$ as follows.

\begin{table}[ht]
\caption{}\label{K2}
\renewcommand\arraystretch{1.5}
\noindent\[
\begin{array}{|c|c|}
\hline
\textrm{Type of $R$} & K^2_S \\
 \hline

<2,1>+<3,2>+<5,4>+ <b;2,1;3,1;5,1> & \frac{30(b-1)^2}{30b-31} \geq \frac{30}{29}\\
%\hline
 <2,1>+<3,2>+<5,2>+ <b;2,1;3,1;5,3> & \frac{6(5b-7)^2}{5(30b-43)}\geq \frac{54}{85} \\
%\hline
 <2,1>+<3,2>+<5,3>+<b;2,1;3,1;5,2> & \frac{6(5b-6)^2}{5(30b-37)} \geq \frac{96}{115}\\
%\hline
<2,1>+<3,2>+<5,1>+ <b;2,1;3,1;5,4> & \frac{6(5b-8)^2}{5(30b-49)} \geq \frac{24}{55}\\
%\hline
<2,1>+<3,1>+<5,4>+ <b;2,1;3,2;5,1> & \frac{10(3b-4)^2}{3(30b-41)} \geq \frac{40}{57}\\
%\hline
<2,1>+<3,1>+<5,2>+ <b;2,1;3,2;5,3> & \frac{2(15b-26)^2}{15(30b-53)} \geq \frac{32}{105}\\
%\hline
 <2,1>+<3,1>+<5,3>+<b;2,1;3,2;5,2> & \frac{2(15b-23)^2}{15(30b-47)} \geq \frac{98}{195}\\
%\hline
<2,1>+<3,1>+<5,1>+<b;2,1;3,2;5,4> & \frac{2(15b-29)^2}{15(30b-59)} \geq \frac{2}{15}\\
\hline
\end{array}
\]
\end{table}

In each case,
$e_{orb}(S)=-1+\frac{1}{2}+\frac{1}{3}+\frac{1}{5}+\frac{1}{120m}
\leq \frac{5}{120}$. From the table we see that $K_S^2>
3e_{orb}(S)$, so $-K_S$ is ample by the orbifold
Bogomolov-Miyaoka-Yau inequality.
\end{proof}

This completes the proof of (1) and (2) of Theorem \ref{class}. To
prove the remaining parts, we need to analyze $(-1)$-curves on the
minimal resolution $S'$. Note that by Lemma \ref{-k} $S'$ contains
no $(-n)$-curve with $n\ge 2$ other than the exceptional curves of
$f:S' \rightarrow S$.

The following proposition will be proved case by case in the next
section.

\begin{proposition}[]\label{mainprop1}
If $S$ has $4$ singularities $p_1, p_2, p_3, p_4$ of type $<2,1>$,
$<3,\alpha>$, $<5,\beta>$, $<b;2,1;3,3-\alpha;5,5-\beta>$, $b\ge
2$, respectively, as in Table 4, then there are three mutually
disjoint $(-1)$-curves $C_1, C_2, C_3$ on $S'$ such that
\begin{enumerate}
\item each $C_i$ intersects exactly $2$ components of $f^{-1}(p_1)\cup f^{-1}(p_2)\cup f^{-1}(p_3)\cup f^{-1}(p_4)$ with multiplicity $1$ each,
\item $C_1$ intersects the component of the branch $<2,1>$ of $f^{-1}(p_4)$ and the component of $f^{-1}(p_1)$,
$C_2$ intersects the terminal component of the branch
$<3,3-\alpha>$ of $f^{-1}(p_4)$ and one end component of
$f^{-1}(p_2)$, and $C_3$ intersects the terminal component of the
branch $<5,5-\beta>$ of $f^{-1}(p_4)$ and one end component of
$f^{-1}(p_3)$ which is a $(-2)$-curve if $\beta=2$ or $4$, a
$(-3)$-curve if $\beta=3$, and a $(-5)$-curve if $\beta=1$.
\end{enumerate}
\end{proposition}

\begin{proposition}[]\label{mainprop2}
\begin{enumerate}
\item The surface $S'$ can be blown down to the Hirzebruch surface
$F_{b-3}$. Conversely, $S'$ can be obtained by starting with
$F_{b-3}$ and blowing up $3$ points lying on a section $s_0$ with
$s_0^2=3-b$.
\item If two rational homology projective planes $S_1$ and $S_2$
have the same type of singularities $<2,1>+<3,\alpha>+<5,\beta>+
<b;2,1;3,3-\alpha;5,5-\beta>$, $b\ge 2$, then $S_1\cong S_2$.
\end{enumerate}
\end{proposition}

\begin{proof}
(1) By Proposition \ref{mainprop1} there are three mutually
disjoint $(-1)$-curves $C_1$, $C_2$, $C_3$ on $S'$ satisfying
$(1)$ and $(2)$ of Proposition \ref{mainprop1}. By starting with
them, we can blow down $S'$ to $F_{b-3}$ so that the image of the
exceptional curves of the map $S'\to F_{b-3}$ consists of three
points lying on a section $s_0$ with $s_0^2=3-b$. The section
$s_0$ is the unique negative section when $b\ge 4$, and is one of
the general sections when $b\le 3$. When $b=2$, $F_{-1}=F_1$.

(2) The blow-up process from $F_{b-3}$ to $S'$ depends on the
choice of three fibres. The choice of three fibres is unique up to
automorphisms of $F_{b-3}$, and the blow-up process from $F_{b-3}$
to $S'$ is uniquely determined by the type of singularities of
$S$.
\end{proof}

This completes the proof of (3) of Theorem \ref{class}.

The following examples mentioned in Introduction were discussed in
\cite{Kol08}, Example 31.

\begin{example}\label{Brieskorn}
Consider the $2m$-ary icosahedral group $$I_m\subset
GL(2,\mathbb{C})$$ of order $120m$ in Brieskorn's list (Table 1).
Let $Z\subset I_m$ be its center, then $Z\cong \mathbb{Z}_{2m}$
and $I_m/Z\cong \mathfrak{A}_5\subset PGL(2,\mathbb{C})$, the
icosahedral group. Extend the natural $I_m$-action on
$\mathbb{C}^2$ to $\mathbb{CP}^2$. The center acts trivially on
the line at infinity and $\mathbb{CP}^2/Z$ is a cone over the
rational normal curve of degree $2m=|Z|$. Then
$$S_{I_m}:=\mathbb{CP}^2/I_{m}=(\mathbb{CP}^2/Z)/\mathfrak{A}_5$$
has 4 quotient singularities, one of type $\mathbb{C}^2/I_{m}$ at
the origin, three of order 2, 3, 5 at infinity. The fundamental
group of $S_{I_m}^0$ is $\mathfrak{A}_5$. By Theorem \ref{class}
(1), the types of the 3 cyclic singularities are determined by the
types of the 3 branches of the non-cyclic singularity. By
Proposition \ref{mainprop1} and \ref{mainprop2}, its minimal
resolution $S_{I_m}'$ can be blown down to the Hirzebruch surface
$F_{b-3}$, where $b$ is determined by $m$ as in Table 1.
Conversely, $S_{I_m}'$ can be obtained by starting with $F_{b-3}$
and blowing up 3 points lying on section $s_0$ with $s_0^2=3-b$.
\end{example}

This completes the proof of (4) and (5) of Theorem \ref{class}.

\section{Proof of Proposition \ref{mainprop1}}

As before, let  $p_1,p_2,p_3,p_4$ be the singular points of $S$ of
order 2, 3, 5, $120m$, respectively, and let $f:S' \rightarrow S$
be a minimal resolution. Let $R_{p_i}$ be the sublattice of
$H^2(S', \mathbb{Z})$ generated by all exceptional curves
contained in $f^{-1}(p_i)$.

Let $C$ be an irreducible curve on $S'$. By Lemma
\ref{coprime}(6), $C$ can be written as $C=kM+r$ for some integer
$k$ and some $r\in R$, hence as
\begin{equation}\label{C}
C = \frac{k}{\sqrt{D}}f^*K_S+ C(1) +C(2) +C(3) +C(4),
\end{equation}
where $C(i)$ is a $\mathbb{Q}$-divisor supported on $f^{-1}(p_i)$
that is of the form
$$C(i)=a_ie_i + r_i$$
for some integer $a_i$ and some $r_i\in R_{p_i}$, where $e_i$ is a
generator of the discriminant group disc$(R_{p_i})$.

\begin{lemma}[]\label{lem4.1} Let $C$ be an irreducible curve on $S'$ of the form
\eqref{C}.
\begin{enumerate}
\item  $C(i)^2=0$ if and only if $C(i)=0$ if and only if $C$ does not meet
$f^{-1}(p_i)$. \item $C(1)^2=-\frac{1}{2}x$ for some integer $x\ge
0$,\\  $C(1)^2=-\frac{1}{2}$ if and only if $C$ meets with
multiplicity $1$ the component of $f^{-1}(p_1)$ .
\item Assume that $p_2$ is of type $<3,2>$. Then\\ $C(2)^2=-\frac{2}{3}y$ for some
integer $y\ge 0$,\\  $C(2)^2=-\frac{2}{3}$ if and only if $C$
meets with multiplicity $1$ exactly one of the two components of
$f^{-1}(p_2)$.
\item Assume that $p_3$ is of type $<5,4>$.
Then\\ $C(3)^2\le -\frac{4}{5}$ if $C(3)\neq 0$,\\
$C(3)^2= -\frac{4}{5}$ if and only if $C$ meets with multiplicity
$1$ exactly one of the two end components of $f^{-1}(p_3)$.
\end{enumerate}
\end{lemma}

\begin{proof} (1) The first equivalence follows from the negative definiteness of exceptional curves.
Note that $EC=EC(i)$ for any curve $E\subset f^{-1}(p_i)$.\\
The curve $C$ does not meet $f^{-1}(p_i)$ iff $EC=0$ for any curve
$E\subset f^{-1}(p_i)$ iff $EC(i)=0$ for any curve $E\subset
f^{-1}(p_i)$ iff $C(i)=0$.

(2) is trivial.

(3) Let $E_1, E_2$ be the exceptional curves generating $R_{p_2}$.
Take $$e:=-\frac{E_1+2E_2}{3}=E_2^*$$ as a generator of {\rm
disc}$(R_{p_2})$. Then $C(2)$ is of the form
$C(2)=ae+b_1E_1+b_2E_2$ for some integers $a, b_1, b_2$, hence of
the form $C(2)=se+tE_2$ for some integers $s,t$. We have
$$C(2)^2=-\frac{2}{3}(s^2-3st+3t^2).$$ It is easy to see that
$y:=s^2-3st+3t^2=(s-3t/2)^2+3t^2/4\ge 0$ for all $s, t$.

$C$ meets exactly one of the two components of $f^{-1}(p_2)$ with
multiplicity $1$ iff $(E_1C(2), E_2C(2))=(1,0)$ or $(0,1)$ iff
$C(2)=E_1^*=2e+E_2$ or $C(2)=E_2^*=e$ iff $(s,t)=(2,1)$ or
$(1,0)$. Both cases satisfy $C(2)^2=-2/3$. Conversely, if
$C(2)^2=-2/3$, then there are six solutions $(s,t)=\pm (1,0), \pm
(2,1), \pm (1,1)$ for the equation $y=(s-3t/2)^2+3t^2/4=1$. Since
$E_iC(2)=E_iC\ge 0$ for $i=1,2$, there remain only two solutions
$(s,t)=(1,0), (2,1)$.

(4) Let $E_1, E_2, E_3, E_4$ be the exceptional curves generating
$R_{p_3}$. Take
$$e:=-\frac{E_1+2E_2+3E_3+4E_4}{5}=E_4^*$$ as a generator of {\rm
disc}$(R_{p_3})$. Then $C(3)$ is of the form
$C(3)=ae+b_1E_1+b_2E_2+b_3E_3+b_4E_4$ for some integers $a, b_1,
b_2, b_3, b_4$, hence of the form $C(3)=se+uE_2+vE_3+wE_4$ for
some integers $s,u,v, w$. We have
$$\begin{array}{lll} C(3)^2 &=&
-\frac{4}{5}s^2-2u^2-2v^2-2w^2+2sw+2uv+2vw \\ &=&
-\frac{4}{5}\{(s-\frac{5w}{4})^2+\frac{5}{2}(u-\frac{v}{2})^2+\frac{15}{8}(v-\frac{2w}{3})^2+\frac{5}{48}w^2\}.
\end{array}$$
To prove the first assertion, assume that
$$(s-\frac{5w}{4})^2+\frac{5}{2}(u-\frac{v}{2})^2+\frac{15}{8}(v-\frac{2w}{3})^2+\frac{5}{48}w^2<1.$$
We need to show that $(s,u,v,w)=(0,0,0,0)$. The above inequality
implies that $w^2\le 9$, i.e., $w=0, \pm 1, \pm 2, \pm 3$. If
$w=0$, then there is only one solution $(s,u,v,w)=(0,0,0,0)$ to
the inequality. If $w=\pm 1, \pm 2, \pm 3$, no solution to the
inequality. This proves the first assertion.

$C$ meets exactly one of the two end components of $f^{-1}(p_3)$
with multiplicity $1$ iff $(E_1C, E_2C, E_3C, E_4C )=(1,0,0,0)$ or
$(0,0,0,1)$ iff $C(3)=E_1^*=4e+E_2+2E_3+3E_4$ or $C(3)=E_4^*=e$
iff $(s,u,v,w)=(4,1,2,3)$ or $(1,0,0,0)$. Both cases satisfy
$C(3)^2=-4/5$. Conversely, if $C(3)^2=-\frac{4}{5}$, then
$$(s-\frac{5w}{4})^2+\frac{5}{2}(u-\frac{v}{2})^2+\frac{15}{8}(v-\frac{2w}{3})^2+\frac{5}{48}w^2=1.$$
There are ten solutions to this equation,

$(s,u,v,w)=\pm (1,0,0,0)$, $\pm (4,1,2,3)$, $\pm (1,1,1,1)$, $\pm
(1,0,1,1)$, $\pm (1,0,0,1)$.\\ Since $E_iC(3)=E_iC\ge 0$ for
$i=1,2,3,4$, there remain only two solutions

$(s,u,v,w)=(4,1,2,3)$, $(1,0,0,0)$.
\end{proof}

\subsection{ Case 1: $<2,1>+<3,2>+<5,4>+ <b;2,1;3,1;5,1>$, $b\ge
2$}

In this case, the number of exceptional curves $\mu=11$, so
$K_{S'}^2=-2$. Let $E_1, \ldots, E_4$ be the components of
$f^{-1}(p_4)$ such that

\medskip
 $$
\begin{picture}(50,35)
\put(0,25){$\overset{-3}E_2-\overset{-b}E_4-\overset{-5}E_3$}
 \put(33,15){\line(0,0){6}}
 \put(28,5){$\underset{-2}{E_1}$}
\end{picture}
$$

\medskip\noindent
is their dual graph. We compute
\begin{equation}\label{ks'1}
K_{S'} = f^*K_S - \frac{(15b-16)E_1 + (20b-21)E_2  + (24b-25)E_3 +
(30b-32)E_4}{30b-31},
\end{equation}
$$K_S^2=\frac{30(b-1)^2}{30b-31},\quad |\det(R)|=30\cdot (30b-31),\quad
D=|\det(R)|K_S^2=30^2(b-1)^2.$$ We also compute the dual vectors,
\medskip

$E_1^*=-\frac{1}{30b-31}\{(15b-8)E_1+5E_2+3E_3+15E_4\}$,

$E_2^*=-\frac{1}{30b-31}\{5E_1+(10b-7)E_2+2E_3+10E_4\}$,

$E_3^*=-\frac{1}{30b-31}\{3E_1+2E_2+(6b-5)E_3+6E_4\}.$

\medskip\noindent
{\bf Claim 4.1.1.} Let $C$ be a $(-1)$-curve of the form
\eqref{C}. Suppose that $C$ meets $f^{-1}(p_4)$. Then it satisfies
one of the following three cases:
$$
\begin{array}{|c|c|c|c|c|c|c|c|c|c|}
\hline
 \textrm{Case} &   CE_4 & CE_3 & CE_2 & CE_1 &  k\\
\hline
({\rm a}) &  0 & 0 &  0 & 1 &  -15\\
\hline
({\rm b}) &  0 & 0   & 1 & 0&  -10\\
\hline
({\rm c}) & 0 & 1 & 0 & 0 &  -6\\
\hline
\end{array}
$$

\begin{proof}
We use the same argument as in the proof of Lemma \ref{spor1}.
First note that
$(f^*K_S)C=\frac{k}{\sqrt{D}}(f^*K_S)^2=\frac{(b-1)k}{30b-31}.$
Since $-K_S$ is ample and $C\notin R$, $(f^*K_S)C<0$, so $k<0$.
Intersecting $C$ with \eqref{ks'1} we get $$C\{(15b-16)E_1 +
(20b-21)E_2 + (24b-25)E_3 + (30b-32)E_4\}=(b-1)k+30b-31.$$ This is
possible only if $C$ satisfies one of the three cases (a), (b),
(c), or the case

\medskip\noindent $({\rm d})\,\,\, CE_4=1, CE_3=0, CE_2=0,
CE_1=0,\quad b=2, \quad k=-1.$

\medskip\noindent In the last case, we compute $C(4)=E_4^*=- \frac{1}{29}(15E_1 +
10E_2 + 6E_3 + 30E_4)$,\\so $C(4)^2=E_4^*C(4)=-\frac{30}{29}$ and
hence we get
$$\underset{j\le 3}{\sum} C(j)^2 =C^2-C(4)^2-(\frac{-1}{30}f^*K_S)^2=-1+\frac{30}{29}
-\frac{1}{30\cdot 29}>0,$$ contradicts  the negative definiteness
of exceptional curves.
\end{proof}

\medskip\noindent
{\bf Claim 4.1.2.} Let $C$ be a $(-1)$-curve of the form
\eqref{C}. Suppose that $C$ meets $f^{-1}(p_4)$. Then $C$ meets
only one component of $f^{-1}(p_1)\cup f^{-1}(p_2)\cup
f^{-1}(p_3)$, the intersection multiplicity is $1$, and the
component is
\begin{enumerate}
\item  the component of $f^{-1}(p_1)$, if $C$ satisfies {\rm (a)},
\item one of the two components of $f^{-1}(p_2)$, if $C$ satisfies {\rm (b)},
\item one of the two end components of $f^{-1}(p_3)$, if $C$ satisfies {\rm (c)}.
\end{enumerate}

\begin{proof}
Assume that $C$ satisfies {\rm (a)}. Then, $C(4)=E_1^*$,
$C(4)^2=E_1^*C(4)=-\frac{15b-8}{30b-31}$,

$C(1)^2 +C(2)^2
+C(3)^2=C^2-C(4)^2-(\frac{-15}{30(b-1)}f^*K_S)^2=-\frac{1}{2}.$\\
By Lemma \ref{lem4.1}, $C(2)=C(3)=0$, $C(1)^2= -\frac{1}{2}$, and
 $C$ does not meet $f^{-1}(p_2)\cup f^{-1}(p_3)$, but meets the
component of $f^{-1}(p_1)$ with multiplicity 1.

 Assume that $C$ satisfies {\rm (b)}. Then, $C(4)=E_2^*$, $C(4)^2=E_2^*C(4)=-\frac{10b-7}{30b-31}$,

$C(1)^2 +C(2)^2
+C(3)^2=C^2-C(4)^2-(\frac{-10}{30(b-1)}f^*K_S)^2=-\frac{2}{3}.$\\
By Lemma \ref{lem4.1}, $C(1)=C(3)=0$, $C(2)^2= -\frac{2}{3}$, and
$C$ does not meet $f^{-1}(p_1)\cup f^{-1}(p_3)$, but meets one of
the two components of $f^{-1}(p_2)$ with multiplicity 1.

Assume that $C$ satisfies {\rm (c)}. Then, $C(4)=E_3^*$,
$C(4)^2=E_3^*C(4)=-\frac{6b-5}{30b-31}$,

$C(1)^2 +C(2)^2
+C(3)^2=C^2-C(4)^2-(\frac{-6}{30(b-1)}f^*K_S)^2=-\frac{4}{5}.$\\
By Lemma \ref{lem4.1}, $C(1)=C(2)=0$, $C(3)^2= -\frac{4}{5}$, and
$C$ does not meet $f^{-1}(p_1)\cup f^{-1}(p_2)$, but meets one of
the end components of $f^{-1}(p_3)$ with multiplicity 1.
\end{proof}

\medskip\noindent
{\bf Claim 4.1.3.} There are three, mutually disjoint,
$(-1)$-curves $C_1, C_2, C_3$ satisfying {\rm (a)}, {\rm (b)},
{\rm (c)} from Claim 4.1.1, respectively.

\begin{proof}
By Lemma \ref{-ample}, $S'$ is a rational surface. Since
$K_{S'}^2<8$, $S'$ contains a $(-1)$-curve and can be blown down
to a minimal rational surface $F_n$ or $\mathbb{CP}^2$.

Assume that there is no $(-1)$-curve $C\subset S'$ meeting
$f^{-1}(p_4)$. Then, since $S'$ cannot contain a $(-l)$-curve with
$l\ge 2$ other than the exceptional curves of $f$ (Lemma
\ref{-k}), the configuration of $f^{-1}(p_4)$ remains the same
under the blow down process to $F_n$ or $\mathbb{CP}^2$. This is
impossible, as the configuration would define a negative definite
sublattice of rank 4 inside the Picard lattice of $F_n$ or
$\mathbb{CP}^2$.

Assume that there is only one $(-1)$-curve meeting $f^{-1}(p_4)$.
Then, the 3 components of $f^{-1}(p_4)$ untouched by the
$(-1)$-curve remain the same under the blow down process and
define a negative definite sublattice of rank 3 inside the Picard
lattice of $F_n$ or $\mathbb{CP}^2$. This is impossible.

If there are only two $(-1)$-curve meeting $f^{-1}(p_4)$. Then the
2 components of $f^{-1}(p_4)$ untouched by the two $(-1)$-curves
would remain the same under the blow down process and define a
negative definite sublattice of rank 2 inside the Picard lattice
of $F_n$ or $\mathbb{CP}^2$. Again, this is impossible.

For the mutual disjointness, we note that

$C_1= \frac{-15}{30(b-1)}f^*K_S+ C_1(1) +E_1^*$,

$C_2= \frac{-10}{30(b-1)}f^*K_S+ C_2(2) +E_2^*$,

$C_3= \frac{-6}{30(b-1)}f^*K_S+ C_3(3) +E_3^*$.\\
A direct calculation shows that $C_iC_j=0$ for $i\neq j$.
\end{proof}

\subsection{Case 2: $<2,1>+<3,2>+<5,2>+ <b;2,1;3,1;5,3>$, $b\ge
2$}

In this case, $\mu=10$, so $K_{S'}^2=-1$. Let $B_1, B_2$ be the
components of $f^{-1}(p_3)$, and $E_1, \ldots, E_5$ be the
components of $f^{-1}(p_4)$ such that

\medskip
 $$
\begin{picture}(40,30)
\put(-60,25){$\overset{-2}B_1-\overset{-3}B_2$}
\put(0,25){$\overset{-3}E_2-\overset{-b}E_5-\overset{-2}E_4-\overset{-3}E_3$}
 \put(33,15){\line(0,0){6}}
 \put(28,5){$\underset{-2}{E_1}$}
\end{picture}
$$

\medskip\noindent
is their dual graph. Then
\begin{equation}\label{ks'2}
    \begin{array}{lll}
 K_{S'} &= &f^*K_S - \frac{1}{5}{(B_1+2B_2)}
-\frac{1}{30b-43}\{(15b-22)E_1 + (20b-29)E_2\\
        & &  + (18b-26)E_3+(24b-35)E_4+(30b-44)E_5\},
    \end{array}
\end{equation}
$$K_S^2=\frac{6(5b-7)^2}{5(30b-43)},\quad |\det(R)|=30\cdot (30b-43),\quad
D=6^2(5b-7)^2.$$ We also compute the dual vectors,
\medskip

$B_1^*=-\frac{3B_1+B_2}{5}\quad \quad B_2^*=-\frac{B_1+2B_2}{5}$,

$E_1^*=-\frac{1}{30b-43}\{(15b-14)E_1+ 5E_2+3E_3+9E_4+15E_5\}$,

$E_2^*=-\frac{1}{30b-43}\{5E_1+(10b-11)E_2+2E_3+6E_4+10E_5\}$,

$E_3^*=-\frac{1}{30b-43}\{3E_1+2E_2+(12b-16)E_3+(6b-5)E_4+6E_5\}$.

\medskip\noindent
{\bf Claim 4.2.1.} Let $C$ be a $(-1)$-curve of the form
\eqref{C}. Suppose that $C$ meets $f^{-1}(p_4)$. Then it satisfies
one of the
  following three cases: $$
\begin{array}{|c|c|c|c|c|c|c|c|c|c|}
\hline
\textrm{ Case} &   CE_5 & CE_4 & CE_3 & CE_2 & CE_1 &  CB_2 & CB_1 & k\\
\hline
({\rm a}) & 0 & 0 & 0 &  0 & 1 & 0 &0 &  -15\\
\hline
({\rm b}) & 0 & 0 & 0   & 1 & 0& 0 & 0 &  -10\\
\hline
({\rm c}) & 0 & 0 & 1 & 0 & 0 & 0 &1 &   -6\\
\hline
\end{array}
$$

\begin{proof} First note that
$(f^*K_S)C=\frac{k}{\sqrt{D}}(f^*K_S)^2=\frac{(5b-7)k}{5(30b-43)}.$
Since $-K_S$ is ample and $C\notin R$, we see that $k<0$.
Intersecting $C$ with \eqref{ks'2} we get\\
$(30b-43)C(B_1+2B_2 ) + 5C\{(15b-22)E_1 + (20b-29)E_2 +
(18b-26)E_3 +
(24b-35)E_4+(30b-44)E_5\} =(5b-7)k+5(30b-43) <5(30b-43).$\\
This is possible only if $C$ satisfies one of the three cases or
the following case
$$({\rm d})\,\,\, CE_5=0, CE_4=1, CE_3=CE_2=CE_1=0, CB_1 = 1, CB_2 = 0,   b=2,   k=-1.$$
In case (d), $C(3)= B_1^*$ and $
C(4)=E_4^*=-\frac{1}{17}(9E_1+6E_2+7E_3+21E_4+18E_5)$, thus
$C(1)^2 +C(2)^2=C^2
-C(3)^2-C(4)^2-(\frac{-1}{18}f^*K_S)^2=-1+\frac{3}{5}+\frac{21}{17}
-\frac{1}{30\cdot 17}
>0,$    contradicts the negative definiteness of exceptional curves.
\end{proof}

\medskip\noindent
{\bf Claim 4.2.2.} Let $C$ be a $(-1)$-curve of the form
\eqref{C}. Suppose that $C$ meets $f^{-1}(p_4)$. Then $C$ meets
only one component of $f^{-1}(p_1)\cup f^{-1}(p_2)\cup
f^{-1}(p_3)$, the intersection multiplicity is $1$, and the
component is
\begin{enumerate}
\item  the component of $f^{-1}(p_1)$, if $C$ satisfies {\rm (a)},
\item one of the two components of $f^{-1}(p_2)$, if $C$ satisfies {\rm (b)},
\item the component $B_1$ of $f^{-1}(p_3)$, if $C$ satisfies {\rm (c)}.
\end{enumerate}

\begin{proof}
Assume that $C$ satisfies {\rm (a)}. Then, $C(3)=0$ and
$C(4)=E_1^*$, so

$C(1)^2
+C(2)^2=C^2-C(4)^2-(\frac{-15}{6(5b-7)}f^*K_S)^2=-\frac{1}{2}.$\\
By Lemma \ref{lem4.1}, $C(2)=0$ and $C(1)^2= -\frac{1}{2}$.

 Assume that $C$ satisfies {\rm (b)}. Then, $C(3)=0$ and
$C(4)=E_2^*$, so

$C(1)^2
+C(2)^2=C^2-C(4)^2-(\frac{-10}{6(5b-7)}f^*K_S)^2=-\frac{2}{3}.$\\
By Lemma \ref{lem4.1}, $C(1)=0$ and $C(2)^2=-\frac{2}{3}$.

Assume that $C$ satisfies {\rm (c)}. Then,
$C(3)=B_1^*=-\frac{3B_1+B_2}{5}$ and $C(4)=E_3^*$, so

$C(1)^2 +C(2)^2=C^2-C(3)^2 -
C(4)^2-(\frac{-6}{6(5b-7)}f^*K_S)^2=0.$\\ By the negative
definiteness, $C(1)= C(2)= 0$.
\end{proof}

\medskip\noindent
By the same proof as in the previous case, we see that there are
three, mutually disjoint, $(-1)$-curves $C_1, C_2, C_3$ satisfying
{\rm (a)}, {\rm (b)}, {\rm (c)} from Claim 4.2.1, respectively.

\subsection{Case 3: $<2,1>+<3,2>+<5,3>+ <b;2,1;3,1;5,2>$, $b\ge 2$}

In this case, $\mu=10$, so $K_{S'}^2=-1$. Let $B_1, B_2$ be the
components of $f^{-1}(p_3)$, and $E_1, \ldots, E_5$ be the
components of $f^{-1}(p_4)$ such that

\medskip
 $$
\begin{picture}(40,30)
\put(-60,20){$\overset{-2}B_1-\overset{-3}B_2$}
\put(0,20){$\overset{-3}E_2-\overset{-b}E_5-\overset{-3}E_4-\overset{-2}E_3$}
 \put(33,10){\line(0,0){6}}
 \put(28,0){$\underset{-2}{E_1}$}
\end{picture}
$$

\medskip\noindent
is their dual graph. Then
\begin{equation}\label{ks'3}
    \begin{array}{lll}
K_{S'} &=& f^*K_S - \frac{1}{5}{(B_1+2B_2)}
-\frac{1}{30b-37}\{(15b-19)E_1 + (20b-25)E_2\\
 & & + (12b-15)E_3 +(24b-30)E_4+(30b-38)E_5\},
    \end{array}
\end{equation}
$$K_S^2=\frac{6(5b-6)^2}{5(30b-37)},\quad |\det(R)|=30\cdot (30b-37),\quad
D=6^2(5b-6)^2.$$
We also compute the dual vectors,
\medskip

$B_1^*=-\frac{3B_1+B_2}{5}\quad \quad B_2^*=-\frac{B_1+2B_2}{5}$,

$E_1^*=-\frac{1}{30b-37}\{(15b-11)E_1+ 5E_2+3E_3+6E_4+15E_5\}$,

$E_2^*=-\frac{1}{30b-37}\{5E_1+(10b-9)E_2+2E_3+4E_4+10E_5\}$,

$E_3^*=-\frac{1}{30b-37}\{3E_1+2E_2+(18b-21)E_3+(6b-5)E_4+6E_5\}$.

\medskip\noindent
{\bf Claim 4.3.1.} Let $C$ be a $(-1)$-curve of the form
\eqref{C}. Suppose that $C$ meets $f^{-1}(p_4)$. Then it satisfies
one of the
  following three cases: $$
\begin{array}{|c|c|c|c|c|c|c|c|c|c|}
\hline
\textrm{Case} &   CE_5 & CE_4 & CE_3 & CE_2 & CE_1 &  CB_2 & CB_1 & k\\
\hline
({\rm a}) & 0 & 0 & 0 &  0 & 1 & 0 &0 &  -15\\
\hline
({\rm b}) & 0 & 0 & 0   & 1 & 0& 0 & 0 &  -10\\
\hline
({\rm c}) & 0 & 0 & 1 & 0 & 0 & 1 & 0 &   -6\\
\hline
\end{array}
$$

\begin{proof}
Since $(f^*K_S)C=\frac{(5b-6)k}{5(30b-37)}<0$, $k<0$.
Intersecting $C$ with \eqref{ks'3} we get\\
$(30b-37)C(B_1+2B_2 ) + 5C\{(15b-19)E_1 + (20b-25)E_2 +
(12b-15)E_3 +
(24b-30)E_4+(30b-38)E_5\}=(5b-6)k+5(30b-37) <5(30b-37).$\\
This is possible only if $C$ satisfies one of the three cases or
the following case
$$({\rm d})\,\,\, CE_5=0, CE_4=0, CE_3=1, CE_2=0, CE_1=0,  CB_1 = 2, CB_2 = 0, \quad
k=-6.$$ In the last case, $C(3)=2B_1^*$ and $C(4)=E_3^*$, so
$C(3)^2 = -\frac{12}{5}$ and $C(4)^2 = -\frac{18b-21}{30b-37}$,
hence $C(1)^2 +C(2)^2=C^2
-C(3)^2-C(4)^2-(\frac{-6}{6(5b-6)}f^*K_S)^2>0,$ which contradicts
the negative definiteness of exceptional curves.
\end{proof}

\medskip\noindent
{\bf Claim 4.3.2.} Let $C$ be a $(-1)$-curve of the form
\eqref{C}. Suppose that $C$ meets $f^{-1}(p_4)$. Then $C$ meets
only one component of $f^{-1}(p_1)\cup f^{-1}(p_2)\cup
f^{-1}(p_3)$, the intersection multiplicity is $1$, and the
component is
\begin{enumerate}
\item  the component of $f^{-1}(p_1)$, if $C$ satisfies {\rm (a)},
\item one of the two components of  $f^{-1}(p_2)$, if $C$ satisfies {\rm (b)},
\item the component $B_2$ of $f^{-1}(p_3)$, if $C$ satisfies {\rm (c)}.
\end{enumerate}

\begin{proof}
Assume that $C$ satisfies {\rm (a)}. Then, $C(3)=0$ and
$C(4)=E_1^*$, so  $C(1)^2
+C(2)^2=-1+\frac{15b-11}{30b-37}-(\frac{-15}{6(5b-6)}f^*K_S)^2=-\frac{1}{2}.$
By Lemma \ref{lem4.1}, $C(2)=0$ and $C(1)^2= -\frac{1}{2}$.

 Assume that $C$ satisfies {\rm (b)}. Then, $C(3)=0$ and $C(4)=E_2^*,$ so  $C(1)^2
+C(2)^2=-1+\frac{10b-9}{30b-37}-(\frac{-10}{6(5b-6)}f^*K_S)^2=-\frac{2}{3}.$
By Lemma \ref{lem4.1}, $C(1)=0$ and $C(2)^2=-\frac{2}{3}$.

Assume that $C$ satisfies {\rm (c)}. Then, $C(3)=B_2^*$ and
$C(4)=E_3^*$, so  $C(1)^2
+C(2)^2=-1+\frac{2}{5}+\frac{18b-21}{30b-37}
-(\frac{-6}{6(5b-6)}f^*K_S)^2=0.$ By the negative definiteness,
$C(1)= C(2)= 0$.
\end{proof}

\medskip\noindent
The same proof as in the previous cases shows that there are
three, mutually disjoint, $(-1)$-curves $C_1, C_2, C_3$ satisfying
{\rm (a)}, {\rm (b)}, {\rm (c)} from Claim 4.3.1, respectively.

\subsection{Case 4: $<2,1>+<3,2>+<5,1>+ <b;2,1;3,1;5,4>$, $b\ge 2$}

In this case, $\mu=11$, so $K_{S'}^2=-2$. Let $B$ be the component
of $f^{-1}(p_3)$, and $E_1, \ldots, E_7$ be the components of
$f^{-1}(p_4)$ such that

\medskip
 $$
\begin{picture}(120,30)
\put(0,25){$\overset{-3}E_2-\overset{-b}E_7-\overset{-2}E_6-\overset{-2}E_5-\overset{-2}E_4-\overset{-2}E_3$}
 \put(33,15){\line(0,0){6}}
 \put(28,5){$\underset{-2}{E_1}$}
\end{picture}
$$

\medskip\noindent
is their dual graph. Then
\begin{equation}\label{ks'4}
    \begin{array}{lll}
K_{S'} &=& f^*K_S - \frac{3}{5}B -\frac{1}{30b-49}\{(15b-25)E_1 +
(20b-33)E_2 + (6b-10)E_3\\
 &&
+(12b-20)E_4+(18b-30)E_5+(24b-40)E_6+(30b-50)E_7\},
    \end{array}
\end{equation}
$$K_S^2=\frac{6(5b-8)^2}{5(30b-49)},\quad |\det(R)|=30\cdot (30b-49),\quad
D=6^2(5b-8)^2.$$ We also compute the dual vectors,
\medskip

$E_1^*=-\frac{1}{30b-49}\{(15b-17)E_1+5E_2+3E_3+6E_4+9E_5+12E_6+15E_7\},$

$E_2^*=-\frac{1}{30b-49}\{5E_1+(10b-13)E_2+2E_3+4E_4+6E_5+8E_6+10E_7\},$

$E_3^*=-\frac{1}{30b-49}\{3E_1+2E_2+(24b-38)E_3+(18b-27)E_4+(12b-16)E_5+(6b-5)E_6+6E_7\}.$

\medskip\noindent
{\bf Claim 4.4.1.} Let $C$ be a $(-1)$-curve of the form
\eqref{C}. Suppose that $C$ meets $f^{-1}(p_4)$. Then it satisfies
one of the
  following three cases: $$
\begin{array}{|c|c|c|c|c|c|c|c|c|c|c|}
\hline
\textrm{Case} & CE_7 & CE_6 & CE_5 & CE_4 & CE_3 & CE_2 & CE_1 & CB  & k\\
\hline
({\rm a}) &  0 & 0 & 0 & 0 &  0 & 0 & 1 &0 & -15\\
\hline
({\rm b}) &  0 & 0 & 0 & 0   & 0 & 1& 0 &0 & -10\\
\hline
({\rm c}) &  0 & 0 & 0 & 0 & 1 & 0 &  0 &1 & -6\\
\hline
\end{array}
$$

\begin{proof}
Since $(f^*K_S)C=\frac{(5b-8)k}{5(30b-49)}<0$,  $k<0$.
Intersecting $C$ with \eqref{ks'4} we get\\
$3(30b-49)CB + 5C\{(15b-25)E_1 + (20b-33)E_2 + (6b-10)E_3 +
(12b-20)E_4+(18b-30)E_5+(24b-40)E_6+(30b-50)E_7\} =(5b-8)k+5(30b-49) <5(30b-49).$\\
This is possible only if $C$ satisfies one of the three cases, or
one of the two cases:
 $$
\begin{array}{|c|c|c|c|c|c|c|c|c|c|c|c|}
\hline
 \textrm{Case} & CE_7 & CE_6 & CE_5 & CE_4 & CE_3 & CE_2 & CE_1 & CB  & k & b\\
\hline
({\rm d}) &  0 & 0 & 0 & 0 &  2 & 0 & 0 & 1 & -1 & 2\\
\hline
({\rm e}) &  0 & 0 & 0 & 1 & 0 & 0& 0 & 1 & -1 & 2\\
\hline
\end{array}
$$
In Case (d), $C(3)=B^*=-\frac{1}{5}B$ and $C(4)=2E_3^*$,
 thus

$C(1)^2 +C(2)^2=C^2
-C(3)^2-C(4)^2-(\frac{-1}{12}f^*K_S)^2=-1+\frac{1}{5}+\frac{40}{11}
-\frac{1}{30\cdot 11}
>0.$\\
In Case (e), $C(3)=-\frac{1}{5}B$ and
$C(4)=E_4^*=-\frac{1}{11}(6E_1+4E_2+9E_3+18E_4+16E_5+14E_6+12E_7)$,
 thus

$C(1)^2 +C(2)^2=C^2
-C(3)^2-C(4)^2-(\frac{-1}{12}f^*K_S)^2=-1+\frac{1}{5}+\frac{18}{11}
-\frac{1}{30\cdot 11}
>0.$\\   Both contradict the negative
definiteness of exceptional curves.
\end{proof}

\medskip\noindent
{\bf Claim 4.4.2.} Let $C$ be a $(-1)$-curve of the form
\eqref{C}. Suppose that $C$ meets $f^{-1}(p_4)$. Then $C$ meets
only one component of $f^{-1}(p_1)\cup f^{-1}(p_2)\cup
f^{-1}(p_3)$, the intersection multiplicity with the component is
$1$, and the component is
\begin{enumerate}
\item  the component of $f^{-1}(p_1)$, if $C$ satisfies {\rm (a)},
\item one of the two components of $f^{-1}(p_2)$, if $C$ satisfies {\rm (b)},
\item the component $B$ of $f^{-1}(p_3)$, if $C$ satisfies {\rm (c)}.
\end{enumerate}

\begin{proof}
Assume that $C$ satisfies {\rm (a)}. Then, $C(3)=0$ and
$C(4)=E_1^*$, so $C(4)^2=-\frac{15b-17}{30b-49}$ and $C(1)^2
+C(2)^2=C^2-C(4)^2-(\frac{-15}{6(5b-8)}f^*K_S)^2=-\frac{1}{2}.$ By
Lemma \ref{lem4.1}, $C(2)=0$ and $C(1)^2= -\frac{1}{2}$.

 Assume that $C$ satisfies {\rm (b)}. Then, $C(3)=0$ and
$C(4)=E_2^*$, so $C(1)^2
+C(2)^2=-1+\frac{10b-13}{30b-49}-(\frac{-10}{6(5b-8)}f^*K_S)^2=-\frac{2}{3}.$
By Lemma \ref{lem4.1}, $C(1)=0$ and $C(2)^2=-\frac{2}{3}$.

Assume that $C$ satisfies {\rm (c)}. Then, $C(3)=-\frac{1}{5}B$
and $C(4)=E_3^*$, so  $C(1)^2
+C(2)^2=-1+\frac{1}{5}+\frac{24b-38}{30b-49}-(\frac{-6}{6(5b-8)}f^*K_S)^2=0.$
By the negative definiteness, $C(1)= C(2)= 0$.
\end{proof}

\medskip\noindent
Similarly, we see that there are three, mutually disjoint,
$(-1)$-curves $C_1, C_2, C_3$ satisfying {\rm (a)}, {\rm (b)},
{\rm (c)} from Claim 4.4.1, respectively.

\subsection{Case 5: $<2,1>+<3,1>+<5,4>+ <b;2,1;3,2;5,1>$, $b\ge 2$}

In this case, $\mu=11$, so $K_{S'}^2=-2$. Let $B$ be the component
of $f^{-1}(p_2)$, and $E_1, \ldots, E_5$ be the components of
$f^{-1}(p_4)$ such that

\medskip
 $$
\begin{picture}(120,30)
\put(0,25){$\overset{-2}E_2-\overset{-2}E_3-\overset{-b}E_5-\overset{-5}E_4$}
 \put(60,15){\line(0,0){6}}
 \put(55,5){$\underset{-2}{E_1}$}
\end{picture}
$$

\medskip\noindent
is their dual graph. Then
\begin{equation}\label{ks'5}
    \begin{array}{lll}
K_{S'} &=& f^*K_S - \frac{1}{3}B -\frac{1}{30b-41}\{(15b-21)E_1 +
(10b-14)E_2\\
& &  + (20b-28)E_3+ (24b-33)E_4 +(30b-42)E_5\},
    \end{array}
\end{equation}
$$K_S^2=\frac{10(3b-4)^2}{3(30b-41)},\quad |\det(R)|=30\cdot (30b-41),\quad
D=10^2(3b-4)^2.$$ We also compute the dual vectors,
\medskip

$E_1^*=-\frac{1}{30b-41}\{(15b-13)E_1+ 5E_2+10E_3+3E_4+15E_5\},$

$E_2^*=-\frac{1}{30b-41}\{5E_1+(20b-24)E_2+(10b-7)E_3+2E_4+10E_5\},$

$E_4^*=-\frac{1}{30b-41}\{3E_1+2E_2+4E_3+(6b-7)E_4+6E_5\}$.

\medskip\noindent
{\bf Claim 4.5.1.} Let $C$ be a $(-1)$-curve of the form
\eqref{C}. Suppose that $C$ meets $f^{-1}(p_4)$. Then it satisfies
one of the
  following three cases: $$
\begin{array}{|c| c|c|c|c|c|c|c|c|}
\hline
\textrm{Case} & CE_5 & CE_4 & CE_3 & CE_2 & CE_1 & CB  & k\\
\hline
({\rm a}) &   0 & 0 & 0 &  0 & 1 & 0 &  -15\\
\hline
({\rm b}) &   0 & 0 & 0   & 1 & 0& 1 &   -10\\
\hline
({\rm c}) &  0 & 1 & 0 & 0 & 0 & 0 &   -6\\
\hline
\end{array}
$$

\begin{proof}
Since $(f^*K_S)C=\frac{(3b-4)k}{3(30b-41)}<0$, $k<0$.
Intersecting $C$ with \eqref{ks'5} we get\\
$(30b-41)CB + 3C\{(15b-21)E_1 +(10b-14)E_2 + (20b-28)E_3 +
(24b-33)E_4 + (30b-42)E_5\}=(3b-4)k+3(30b-41) <3(30b-41).$\\ This
is possible only if $C$ satisfies one of the three cases, or one
of the following three cases: $$
\begin{array}{|c|c|c|c|c|c|c|c|c|c| c|}
\hline
\textrm{Case} & CE_6 & CE_5 & CE_4 & CE_3 & CE_2 & CE_1 & CB  & k & b\\
\hline
({\rm d}) &  0 & 0 & 0 & 1 &  0 & 0 & 1 & -1  & 2\\
\hline
({\rm e}) &  0 & 0 &     0 & 0& 2 & 0 & 1 & -1  & 2\\
\hline
({\rm f}) &  0 & 0 &   0 & 0 & 1 & 1 & 0 & -6 & 2\\
\hline
\end{array}
$$
 In Case (d),
$C(2)=-\frac{1}{3}B$ and
$C(4)=E_3^*=-\frac{1}{19}(10E_1+13E_2+26E_3+4E_4+20E_5)$, thus
$C(1)^2
+C(3)^2=C^2-C(2)^2-C(4)^2-(\frac{-1}{20}f^*K_S)^2=-1+\frac{1}{3}+\frac{26}{19}
-\frac{1}{30\cdot 19}
>0.$\\ In Case (e), $C(2)=-\frac{1}{3}B$ and
$C(4)=2E_2^*$, thus\\ $C(1)^2 +C(3)^2=-1+\frac{1}{3}+\frac{64}{19}
-\frac{1}{30\cdot 19}>0.$\\ In Case (f), $C(2)=0$ and
$C(4)=E_1^*+E_2^*$, thus $C(1)^2 +C(3)^2=-1+\frac{43}{19}
-\frac{36}{30\cdot 19}>0.$ All these cases lead to a
contradiction.
\end{proof}

\medskip\noindent
{\bf Claim 4.5.2.} Let $C$ be a $(-1)$-curve of the form
\eqref{C}. Suppose that $C$ meets $f^{-1}(p_4)$. Then $C$ meets
only one component of $f^{-1}(p_1)\cup f^{-1}(p_2)\cup
f^{-1}(p_3)$, the intersection multiplicity with the component is
$1$, and the component is
\begin{enumerate}
\item  the component of $f^{-1}(p_1)$, if $C$ satisfies {\rm (a)},
\item the component of $B$ of $f^{-1}(p_2)$, if $C$ satisfies {\rm (b)},
\item one of the two end components of $f^{-1}(p_3)$, if $C$ satisfies {\rm (c)}.
\end{enumerate}

\begin{proof}
Assume that $C$ satisfies {\rm (a)}. Then, $C(2)=0$ and
$C(4)=E_1^*$, so $C(4)^2=-\frac{15b-13}{30b-41}$, hence $C(1)^2
+C(3)^2=C^2-C(4)^2-(\frac{-15}{10(3b-4)}f^*K_S)^2=-\frac{1}{2}.$
By Lemma \ref{lem4.1}, $C(3)=0$ and $C(1)^2= -\frac{1}{2}$.

 Assume that $C$ satisfies {\rm (b)}. Then, $C(2)=-\frac{1}{3}B$ and
$C(4)=E_2^*$, so  $C(1)^2
+C(3)^2=-1+\frac{1}{3}+\frac{20b-24}{30b-41}-(\frac{-10}{10(3b-4)}f^*K_S)^2=0.$
By the negative definiteness, $C(1)=C(3)=0$.

Assume that $C$ satisfies {\rm (c)}. Then, $C(2)=0$ and
$C(4)=E_4^*$, so  $C(1)^2
+C(3)^2=-1+\frac{6b-7}{30b-41}-(\frac{-6}{10(3b-4)}f^*K_S)^2=-\frac{4}{5}.$
By Lemma \ref{lem4.1}, $C(1)=0$ and $C(3)^2= -\frac{4}{5}$.
\end{proof}

\medskip\noindent
Similarly, we see that there are three, mutually disjoint,
$(-1)$-curves $C_1, C_2, C_3$ satisfying {\rm (a)}, {\rm (b)},
{\rm (c)} from Claim 4.5.1, respectively. In this case, $C_1=
\frac{-15}{10(3b-4)}f^*K_S+ C_1(1) +E_1^*$, $C_2=
\frac{-10}{10(3b-4)}f^*K_S+ C_2(2) +E_2^*$, $C_3=
\frac{-6}{10(3b-4)}f^*K_S+ C_3(3) +E_4^*$.

\subsection{Case 6: $<2,1>+<3,1>+<5,2>+ <b;2,1;3,2;5,3>$, $b\ge 2$}

In this case, $\mu=10$, so $K_{S'}^2=-1$. Let $B$ be the component
of $f^{-1}(p_2)$, $B_2,B_3$ be the components of $f^{-1}(p_3)$,
and $E_1, \ldots, E_6$ be the components of $f^{-1}(p_4)$ such
that

\medskip
 $$
\begin{picture}(40,30)
\put(-60,25){$\overset{-2}B_2-\overset{-3}B_3$}
\put(0,25){$\overset{-2}E_2-\overset{-2}E_3-\overset{-b}E_6-\overset{-2}E_5-\overset{-3}E_4$}
 \put(60,15){\line(0,0){6}}
 \put(55,5){$\underset{-2}{E_1}$}
\end{picture}
$$

\medskip\noindent
is their dual graph. Then
\begin{equation}\label{ks'6}
    \begin{array}{lll}
K_{S'} &=& f^*K_S - \frac{1}{3}B - \frac{1}{5}{(B_2+2B_3)}
  -\frac{1}{30b-53}\{(15b-27)E_1 + (10b-18)E_2\\
&& + (20b-36)E_3 + (18b-32)E_4+(24b-43)E_5+(30b-54)E_6\},
    \end{array}
\end{equation}
$$K_S^2=\frac{2(15b-26)^2}{15(30b-53)},\quad |\det(R)|=30\cdot (30b-53),\quad
D=2^2(15b-26)^2.$$ We also compute the dual vectors,
\medskip

$B_2^*=-\frac{3B_2+B_3}{5}\quad \quad B_3^*=-\frac{B_2+2B_3}{5}$,

$E_1^*=-\frac{1}{30b-53}\{(15b-19)E_1+5E_2+10E_3+3E_4+9E_5+15E_6\},$

$E_2^*=-\frac{1}{30b-53}\{5E_1+(20b-32)E_2+(10b-11)E_3+2E_4+6E_5+10E_6\},$

$E_4^*=-\frac{1}{30b-53}\{3E_1+2E_2+4E_3+(12b-20)E_4+(6b-7)E_5+6E_6\}$.

\medskip\noindent
{\bf Claim 4.6.1.} Let $C$ be a $(-1)$-curve of the form
\eqref{C}. Suppose that $C$ meets $f^{-1}(p_4)$. Then it satisfies
one of the
  following three cases: $$
\begin{array}{|c|c|c|c|c|c|c|c|c|c|c|c|}
\hline
\textrm{Case} & CE_6 & CE_5 & CE_4 & CE_3 & CE_2 & CE_1 & CB_3 & CB_2 & CB & k\\
\hline
({\rm a}) &  0 & 0 & 0 & 0 &  0 & 1 & 0 &0 & 0 & -15\\
\hline
({\rm b}) &  0 & 0 & 0 & 0   & 1 & 0& 0 & 0 & 1 & -10\\
\hline
({\rm c}) &  0 & 0 & 1 & 0 & 0 & 0 & 0 & 1 & 0 & -6\\
\hline
\end{array}
$$

\begin{proof}
Since $(f^*K_S)C=\frac{(15b-26)k}{15(30b-53)}<0$, $k<0$.
Intersecting $C$ with \eqref{ks'6} we get\\
$(30b-53)C(5B+3B_2+6B_3 ) + 15C\{(15b-27)E_1 + (10b-18)E_2 +
(20b-36)E_3\\ + (18b-32)E_4+(24b-43)E_5+(30b-54)E_6\}=(15b-26)k+15(30b-53).$\\
This is possible only if $C$ satisfies one of the three cases, or
one of the following five cases:
$$
\begin{array}{|c|c|c|c|c|c|c|c|c|c|c|c|c|}
\hline
\textrm{Case} & CE_6 & CE_5 & CE_4 & CE_3 & CE_2 & CE_1 & CB_3 & CB_2 & CB & k & b\\
\hline
({\rm d}) &  0 & 0 & 0 & 0 &  1 & 0 & 0 &3 & 0 & -3 & 2\\
\hline
({\rm e}) &  0 & 0 & 0 & 0   & 1 & 0& 1 & 1 & 0 & -3 & 2\\
\hline
({\rm f}) &  0 & 0 & 0 & 0 & 0 & 1 & 0 & 1 & 1 & -1 & 2\\
\hline
({\rm g}) &  0 & 0 & 0 & 0 & 2 & 0 & 0 & 1 & 0 & -6 & 2\\
\hline
({\rm h}) &  0 & 0 & 0 & 1 & 0 & 0 & 0 & 1 & 0 & -6 & 2\\
\hline
\end{array}
$$
In Case (d), $C(2)=0, C(3) =3B_2^*$ and
$C(4)=E_2^*$, thus\\
$C(1)^2 = C^2
-C(3)^2-C(4)^2-(\frac{-3}{8}f^*K_S)^2=-1+\frac{27}{5}+\frac{8}{7}
-\frac{9}{30\cdot 7} >0.$\\ In Case (e), $C(2)=0, C(3)=B_2^*+B_3^*
=-\frac{4B_2+3B_3}{5}$ and
$C(4)=E_2^*$, thus\\
$C(1)^2 = C^2
-C(3)^2-C(4)^2-(\frac{-3}{8}f^*K_S)^2=-1+\frac{7}{5}+\frac{8}{7}
-\frac{9}{30\cdot 7} >0.$\\ In Case (f), $C(2)=-\frac{1}{3}B,
C(3)=B_2^*$ and $C(4)=E_1^*$,
thus\\
$C(1)^2 = C^2
-C(2)^2-C(3)^2-C(4)^2-(\frac{-1}{8}f^*K_S)^2=-1+\frac{1}{3}+\frac{3}{5}+\frac{11}{7}
-\frac{1}{30\cdot 7} >0.$\\ In Case (g), $C(2)=0, C(3)=B_2^*$ and
$C(4)=2E_2^*$,
thus\\
$C(1)^2 = C^2
-C(3)^2-C(4)^2-(\frac{-6}{8}f^*K_S)^2=-1+\frac{3}{5}+\frac{32}{7}
-\frac{36}{30\cdot 7} >0.$\\ In Case (h), $C(2)=0, C(3)=B_2^*$ and
$C(4)=E_3^*=-\frac{1}{7}(10E_1+9E_2+18E_3+4E_4+12E_5+20E_6)$,
thus\\
$C(1)^2 = C^2
-C(3)^2-C(4)^2-(\frac{-6}{8}f^*K_S)^2=-1+\frac{3}{5}+\frac{18}{7}
-\frac{36}{30\cdot 7} >0.$\\
 All contradict the negative
definiteness of exceptional curves.
\end{proof}

\medskip\noindent
{\bf Claim 4.6.2.} Let $C$ be a $(-1)$-curve of the form
\eqref{C}. Suppose that $C$ meets $f^{-1}(p_4)$. Then $C$ meets
only one component of $f^{-1}(p_1)\cup f^{-1}(p_2)\cup
f^{-1}(p_3)$, the intersection multiplicity with the component is
$1$, and the component is
\begin{enumerate}
\item  the component of $f^{-1}(p_1)$, if $C$ satisfies {\rm (a)},
\item the component $B$ of $f^{-1}(p_2)$, if $C$ satisfies {\rm (b)},
\item the component $B_2$ of $f^{-1}(p_3)$, if $C$ satisfies {\rm (c)}.
\end{enumerate}

\begin{proof}
Assume that $C$ satisfies {\rm (a)}. Then, $C(2)=C(3)=0$ and
$C(4)=E_1^*$, so $C(4)^2=-\frac{15b-19}{30b-53}$, hence $C(1)^2
=C^2-C(4)^2-(\frac{-15}{2(15b-26)}f^*K_S)^2=-\frac{1}{2}.$

Assume that $C$ satisfies {\rm (b)}. Then, $C(2)=-\frac{1}{3}B$,
$C(3) = 0$ and $C(4)=E_2^*,$ so  $C(1)^2
=-1+\frac{1}{3}+\frac{20b-32}{30b-53}-(\frac{-10}{2(15b-26)}f^*K_S)^2=0.$
Hence $C(1)=0$.

Assume that $C$ satisfies {\rm (c)}. In this case, $C(2)=0$,
$C(3)=B_2^*$ and $C(4)=E_4^*,$ so  $C(1)^2
=-1+\frac{3}{5}+\frac{12b-20}{30b-53}-(\frac{-6}{2(15b-26)}f^*K_S)^2=0.$
Hence $C(1)=0$.
\end{proof}

\medskip\noindent
Similarly, we see that there are three, mutually disjoint,
$(-1)$-curves $C_1, C_2, C_3$ satisfying {\rm (a)}, {\rm (b)},
{\rm (c)} from Claim 4.6.1, respectively.

\subsection{Case 7: $<2,1>+<3,1>+<5,3>+ <b;2,1;3,2;5,2>$, $b\ge 2$}
In this case, $\mu=10$, so $K_{S'}^2=-1$. Let $B$ be the component
of $f^{-1}(p_2)$, $B_2,B_3$ be the components of $f^{-1}(p_3)$,
and $E_1, \ldots, E_6$ be the components of $f^{-1}(p_4)$ such
that

\medskip
 $$
\begin{picture}(40,30)
\put(-60,20){$\overset{-2}B_2-\overset{-3}B_3$}
\put(0,20){$\overset{-2}E_2-\overset{-2}E_3-\overset{-b}E_6-\overset{-3}E_5-\overset{-2}E_4$}
 \put(60,10){\line(0,0){6}}
 \put(55,0){$\underset{-2}{E_1}$}
\end{picture}
$$

\medskip\noindent
is their dual graph. Then
\begin{equation}\label{ks'7}
    \begin{array}{lll}
K_{S'} &=& f^*K_S - \frac{1}{3}B - \frac{1}{5}{(B_2+2B_3)}-\frac{1}{30b-47}\{(15b-24)E_1 + (10b-16)E_2 \\
&& + (20b-32)E_3+ (12b-19)E_4+(24b-38)E_5+(30b-48)E_6\},
    \end{array}
\end{equation}
$$K_S^2=\frac{2(15b-23)^2}{15(30b-47)},\quad |\det(R)|=30\cdot (30b-47),\quad
D=2^2(15b-23)^2.$$ We also compute the dual vectors,
\medskip

$B_2^*=-\frac{3B_2+B_3}{5}\quad \quad B_3^*=-\frac{B_2+2B_3}{5}$,

$E_1^*=-\frac{1}{30b-47}\{(15b-16)E_1+5E_2+10E_3+3E_4+6E_5+15E_6\},$

$E_2^*=-\frac{1}{30b-47}\{5E_1+(20b-28)E_2+(10b-9)E_3+2E_4+4E_5+10E_6\},$

$E_4^*=-\frac{1}{30b-47}\{3E_1+2E_2+4E_3+(18b-27)E_4+(6b-7)E_5+6E_6\}$.

\medskip\noindent
{\bf Claim 4.7.1.} Let $C$ be a $(-1)$-curve of the form
\eqref{C}. Suppose that $C$ meets $f^{-1}(p_4)$. Then it satisfies
one of the
  following three cases: $$
\begin{array}{|c|c|c|c|c|c|c|c|c|c|c|c|}
\hline
\textrm{Case} & CE_6 & CE_5 & CE_4 & CE_3 & CE_2 & CE_1 & CB_3 & CB_2 & CB & k\\
\hline
({\rm a}) &  0 & 0 & 0 & 0 &  0 & 1 & 0 &0 & 0 & -15\\
\hline
({\rm b}) &  0 & 0 & 0 & 0   & 1 & 0& 0 & 0 & 1 & -10\\
\hline
({\rm c}) &  0 & 0 & 1 & 0 & 0 & 0 & 1 & 0 & 0 & -6\\
\hline
\end{array}
$$

\begin{proof}
Since $(f^*K_S)C=\frac{(15b-23)k}{15(30b-47)}<0$, $k<0$.
Intersecting $C$ with \eqref{ks'7} we get\\
$(30b-47)C(5B+3B_2+6B_3 ) + 15C\{(15b-24)E_1 + (10b-16)E_2 +
(20b-32)E_3+ (12b-19)E_4+(24b-38)E_5+(30b-48)E_6\}=(15b-23)k+15(30b-47).$\\
This is possible only if $C$ satisfies one of the three cases, or the case \\
$({\rm d})\,\,\,   CE_6 = CE_5=0, CE_4=1, CE_3=0, CE_2=1, CE_1=0,
CB_3 = 0$, $CB_2=1$, $CB = 0,$
  $\quad b=2, \quad k=-3. $\\
 In the last case, $C(2)
= 0, C(3) = B_2^*$ and $C(4) = E_2^* + E_4^*$, thus

$C(1)^2
=C^2-C(3)^2-C(4)^2-(\frac{-3}{14}f^*K_S)^2=-1+\frac{3}{5}+\frac{25}{13}
-\frac{9}{30\cdot 13} >0,$\\ which contradicts the negative
definiteness of exceptional curves.
\end{proof}

\medskip\noindent
{\bf Claim 4.7.2.} Let $C$ be a $(-1)$-curve of the form
\eqref{C}. Suppose that $C$ meets $f^{-1}(p_4)$.
 Then $C$ meets only one component of $f^{-1}(p_1)\cup
f^{-1}(p_2)\cup f^{-1}(p_3)$, the intersection multiplicity with the
component is $1$, and the component is
\begin{enumerate}
\item  the component of $f^{-1}(p_1)$, if $C$ satisfies {\rm (a)},
\item the component $B$ of $f^{-1}(p_2)$, if $C$ satisfies {\rm (b)},
\item the component $B_3$ of $f^{-1}(p_3)$, if $C$ satisfies {\rm (c)}.
\end{enumerate}

\begin{proof}
Assume that $C$ satisfies {\rm (a)}. Then, $C(2)=C(3)=0$ and
$C(4)=E_1^*$, so $C(4)^2=-\frac{15b-16}{30b-53}$, hence
$C(1)^2=C^2-C(4)^2-(\frac{-15}{2(15b-23)}f^*K_S)^2=-\frac{1}{2}.$

 Assume that $C$ satisfies {\rm (b)}. Then, $C(2)=-\frac{1}{3}B, C(3) = 0$ and
$C(4)=E_2^*$, so
$C(1)^2=-1+\frac{1}{3}+\frac{20b-28}{30b-47}-(\frac{-10}{2(15b-23)}f^*K_S)^2=0.$
Hence $C(1)=0$.

Assume that $C$ satisfies {\rm (c)}. In this case, $C(2)=0$,
$C(3)=B_3^*$ and $C(4)=E_4^*$, so
$C(1)^2=-1+\frac{2}{5}+\frac{18b-27}{30b-47}
-(\frac{-6}{2(15b-23)}f^*K_S)^2=0.$ Hence $C(1)=0$.
\end{proof}

\medskip\noindent
Similarly, we see that there are three, mutually disjoint,
$(-1)$-curves $C_1, C_2, C_3$ satisfying {\rm (a)}, {\rm (b)},
{\rm (c)} from Claim 4.7.1, respectively.

\subsection{Case 8: $<2,1>+<3,1>+<5,1>+ <b;2,1;3,2;5,4>$, $b\ge 2$}

In this case, $\mu=11$, so $K_{S'}^2=-2$. Let $B, B_2$ be the
components of $f^{-1}(p_2), f^{-1}(p_3)$, and $E_1, \ldots, E_8$
be the components of $f^{-1}(p_4)$ such that

\medskip
 $$
\begin{picture}(120,30)
\put(0,25){$\overset{-2}E_2-\overset{-2}E_3-\overset{-b}E_8-\overset{-2}E_7-\overset{-2}E_6-\overset{-2}E_5-\overset{-2}E_4$}
 \put(60,15){\line(0,0){6}}
 \put(55,5){$\underset{-2}{E_1}$}
\end{picture}
$$

\medskip\noindent
is their dual graph. Then
\begin{equation}\label{ks'8}
\begin{array}{lll}
K_{S'} &=& f^*K_S - \frac{1}{3}B - \frac{3}{5}B_2
-\frac{b-2}{30b-59}(15E_1 + 10E_2 + 20E_3 +
6E_4\\
& & +12E_5+18E_6+24E_7+30E_8),
\end{array}
\end{equation}
$$K_S^2=\frac{2(15b-29)^2}{15(30b-59)},\quad |\det(R)|=30\cdot (30b-59),\quad
D=2^2(15b-29)^2.$$ We also compute the dual vectors,
\medskip

$E_1^*=-\frac{1}{30b-59}\{(15b-22)E_1+5E_2+10E_3+3E_4+6E_5+9E_6+12E_7+15E_8\},$

$E_2^*=-\frac{1}{30b-59}\{5E_1+(20b-36)E_2+(10b-13)E_3+2E_4+4E_5+6E_6+8E_7+10E_8\},$

$E_4^*=-\frac{1}{30b-59}\{3E_1+2E_2+4E_3+(24b-46)E_4+(18b-33)E_5+(12b-20)E_6+(6b-7)E_7+6E_8\}$.

\medskip\noindent
{\bf Claim 4.8.1.} Let $C$ be a $(-1)$-curve of the form
\eqref{C}. Suppose that $C$ meets $f^{-1}(p_4)$. Then it satisfies
one of the
 following three cases: $$
\begin{array}{|c|c|c|c|c|c|c|c|c|c|c|c|}
\hline \textrm{Case} & CE_8 & CE_7 & CE_6 & CE_5 & CE_4 & CE_3 &
CE_2 & CE_1 & CB_2
& CB & k\\
\hline
({\rm a}) & 0 & 0 & 0 & 0 & 0 & 0 & 0 & 1 & 0 & 0 & -15\\
\hline
({\rm b}) & 0 & 0 & 0 & 0 & 0 & 0 & 1 & 0 & 0 & 1 & -10\\
\hline
({\rm c}) & 0 & 0 & 0 & 0 & 1 & 0 & 0 & 0 & 1 & 0 & -6\\
\hline
\end{array}
$$

\begin{proof}
Since $(f^*K_S)C=\frac{(15b-29)k}{15(30b-59)}<0$, $k<0$.
Intersecting $C$ with \eqref{ks'8} we get\\ $(30b-59)C(5B+9B_2 ) +
15(b-2)C\{15E_1 + 10E_2 + 20E_3 +
6E_4+12E_5+18E_6+24E_7+30E_8\}=(15b-29)k+15(30b-59) <15(30b-59).$\\
This is possible only if $C$ satisfies one of the three cases, or
the case

\medskip\noindent $({\rm d}) \,\,\, CB_2 = CB =1, b = 2,
k = -1,\,\,\, (CE_i\,\,\,{\rm are}\,\,\,{\rm not}\,\,\,{\rm
determined}).$

\medskip\noindent In case (d), $C(2) =
-\frac{1}{3}B$ and $C(3) = -\frac{1}{5}B_2$, thus

$C(1)^2 + C(4)^2 = C^2-C(2)^2-C(3)^2-(\frac{-1}{2}f^*K_S)^2 =
-\frac{1}{2}.$\\ Also note that in this case the sublattice
$R_{p_4}\subset H^2(S', \mathbb{Z})$ generated by the components
of $f^{-1}(p_4)$ is a negative definite unimodular lattice of rank
8. In particular, $R_{p_4}^*=R_{p_4}$, so $C(4)\in R_{p_4}$ and
$C(4)^2$ is a non-positive even integer. By Lemma \ref{lem4.1},
$C(4)^2 = 0$. Thus $C$ does not meet $f^{-1}(p_4)$, contradicts
the assumption.
\end{proof}

\medskip\noindent
{\bf Claim 4.8.2.} Let $C$ be a $(-1)$-curve of the form
\eqref{C}. Suppose that $C$ meets $f^{-1}(p_4)$. Then $C$ meets
only one component of $f^{-1}(p_1)\cup f^{-1}(p_2)\cup
f^{-1}(p_3)$, the intersection multiplicity with the component is
$1$, and the component is
\begin{enumerate}
\item  the component of $f^{-1}(p_1)$, if $C$ satisfies {\rm (a)},
\item the component $B$  of $f^{-1}(p_2)$, if $C$ satisfies {\rm (b)},
\item the component $B_2$ of $f^{-1}(p_3)$, if $C$ satisfies {\rm (c)}.
\end{enumerate}

\begin{proof}
Assume that $C$ satisfies {\rm (a)}. Then, $C(2)=C(3)=0$ and
$C(4)=E_1^*$, so $C(4)^2=-\frac{15b-22}{30b-59}$, hence $C(1)^2
=C^2-C(4)^2-(\frac{-15}{2(15b-29)}f^*K_S)^2=-\frac{1}{2}.$

 Assume that $C$ satisfies {\rm (b)}.  Then, $C(2)=-\frac{1}{3}B, C(3) = 0$ and
$C(4)=E_2^*$, so
$C(1)^2=-1+\frac{1}{3}+\frac{20b-36}{30b-59}-(\frac{-10}{2(15b-29)}f^*K_S)^2=0.$
Hence $C(1)=0$.

 Assume that $C$ satisfies {\rm (c)}.  Then, $C(2) =
0, C(3)=-\frac{1}{5}B_2$ and $C(4)=E_4^*$, so
$C(1)^2=-1+\frac{1}{5}+\frac{24b-46}{30b-59}-(\frac{-6}{2(15b-29)}f^*K_S)^2=0.$
Hence $C(1)=0$.
\end{proof}

\medskip\noindent
Similarly, we see that there are three, mutually disjoint,
$(-1)$-curves $C_1, C_2, C_3$ satisfying {\rm (a)}, {\rm (b)},
{\rm (c)} from Claim 4.8.1, respectively.

\bigskip
{\bf Acknowledgements.} We thank J\'anos Koll\'ar for many useful
comments. We are grateful to the referee for helpful suggestions
which have led to improvements in the presentation of the
manuscript.

%%%%%%%%%%%%%%%%%%%%%%%%%%%%%%%%%%%%%%%%%%%%%%%%%%%%%%%%%%%%%%%%%%%%%%%%
%\bibliographystyle{amsplain}

\end{document}